\documentclass[a4paper]{amsart}
\usepackage{graphicx,mathrsfs,amssymb,amsmath,color,hyperref}
\usepackage{accents}
\usepackage[latin1]{inputenc}
\def\reftext{}
\def\eqncr{\\}

\newtheorem{thm}{Theorem}[section]
\newtheorem{proposition}[thm]{Proposition}
\newtheorem{corollary}[thm]{Corollary}
\newtheorem{theorem}[thm]{Theorem}
\newtheorem{lemma}[thm]{Lemma}
\theoremstyle{definition}
\newtheorem{definition}[thm]{Definition}
\newtheorem{remark}[thm]{Remark}

\DeclareMathOperator{\opvtex}{op}

\newcommand{\dbar}{d\hspace{-0.08em}\bar{}\hspace{0.1em}}
\newcommand{\Hcirc}{\mathring{H}}
\newcommand{\cHcirc}{\mathring{\mathcal{H}}}
\newcommand{\cKcirc}{\mathring{\mathcal{K}}}
\newcommand{\revtex}{\operatorname{Re}}
\numberwithin{equation}{section}

\begin{document}
\title[Bounded $H_{\infty}$-calculus  on conic manifolds with boundary]%
{Bounded $H_{\infty}$-calculus for Boundary Value Problems on Manifolds with Conical Singularities}
\author{Nikolaos Roidos}
\address{Department of Mathematics, University of Patras, Rio Patras, Greece}
\email{roidos@math.upatras.gr}
\author{Elmar Schrohe}
\address{Institut f\"ur Analysis, Leibniz Universit\"at Hannover, Germany}
\email{schrohe@math.uni-hannover.de}
\thanks{N. Roidos and E. Schrohe were supported by Deutsche Forschungsgemeinschaft, grant SCHR 319/9-1, in the priority program {\em Geometry at Infinity.}}

\author{J\"org Seiler}
\address{Dipartimento di Matematica, Universit\`a degli Studi di Torino, Turin, Italy}
\email{joerg.seiler@unito.it}
\date{}

\begin{abstract} 
Realizations of differential operators subject to differential boundary conditions on manifolds 
with conical singularities are shown to have a bounded $H_\infty$-calculus in appropriate 
$L_p$-Sobolev spaces provided suitable conditions of parameter-ellipticity are satisfied. 
Applications concern the Dirichlet and Neumann Laplacian and the porous medium equation. 
\end{abstract}

\subjclass[2010]{}

\maketitle
\setcounter{tocdepth}{1}
\tableofcontents

\section{Introduction}
\label{sec:intro}

In this article, we study the $H_{\infty }$-calculus of parameter-elliptic
boundary value problems on manifolds with conical singularities on the
boundary, following up on earlier work in \cite{BSS,CSS0,DSS,SS2}. Moreover,
we present a new way of determining their realizations. As an application
we treat the porous medium equation.

A manifold of dimension $n+1$ with conical singularities on the boundary
is a compact topological space $D$, which contains finitely many points
$\{d_{1},\ldots ,d_{N}\}$ such that

\begin{itemize}
\item[(1)] $D_{\mathrm{reg}}:=D\setminus \{d_{1},\ldots ,d_{N}\}$ is an
$(n+1)$-dimensional smooth manifold with boundary,
\item[(2)] each point $d_{j}$ has a neighborhood which is homeomorphic
to a cone $C_{j} = [0,1)\times Y_{j}/\{0\} \times Y_{j}$ with an $n$-dimensional
manifold with boundary $Y_{j}$.
\end{itemize}

We replace $C_{1}\cup \ldots \cup C_{N}$ by the cylinder
$[0,1)\times Y$, where $Y$ is the disjoint union of the $Y_{j}$, obtaining
a space denoted by ${\mathbb{D}}$. Note that $D_{\mathrm{reg}}$ can be identified
with
${\mathbb{D}}_{\mathrm{reg}}:={\mathbb{D}}\setminus (\{0\}\times Y)$. The
boundary
${\mathbb{B}}_{\mathrm{reg}}:=\partial {\mathbb{D}}_{\mathrm{reg}}$ is the
regular part of an $n$-dimensional conic manifold $($without boundary$)$
${\mathbb{B}}$ which contains the cylinder $[0,1)\times \partial Y$. On
the cylindrical parts, we use variables $(x,y)$ with $x\in [0,1)$ and
$y\in Y$. For more details we refer the reader to
\cite[Section 1.1]{SS_1} or \cite[Section 3]{LS}.

Vector bundles over ${\mathbb{D}}$ are assumed to be smooth over
${\mathbb{D}}_{\mathrm{reg}}$. On the cylindrical part a vector bundle
$E$ over ${\mathbb{D}}$ is the pull-back of a smooth vector bundle
$E_{0}$ over $Y$ under the canonical projection $(x,y)\mapsto y$; the same
applies to vector bundles over ${\mathbb{B}}$. All bundles are supposed
to carry a hermitian structure compatible with the product structure on
the cylindrical part.

A cone differential operator on ${\mathbb{D}}$ acting on sections of a bundle
$E$ $($or, more generally, between sections of two possibly different bundles$)$
is a differential operator with smooth coefficients on the regular part,
while, on the cylindrical part, it has the form
\begin{equation}
\label{eq:A}
A=x^{-\mu }\sum _{j=0}^{\mu }a_{j}(x)(-x\partial _{x})^{j}, \qquad \mu =
\text{order of }A,
\end{equation}
where each $a_{j}(x)$ is a family of differential operators of order
$\mu -j$ on $Y$ acting on sections of $E_{0}$, smooth up to $x=0$. Initially,
we consider $A$ as a map in
${\mathscr{C}}^{\infty ,\infty }({\mathbb{D}},E)$, the space of smooth sections
of $E$ that vanish to infinite order in the tip $x=0$. To give an example,
let $\mathbb{D}_{\textrm{reg}}$ be endowed with a Riemannian metric which,
on the cylindrical part, has the form
\begin{eqnarray}
\label{eq:metric}
g=dx^{2}+x^{2}h(x),
\end{eqnarray}
with a family $h(x)$ of Riemannian metrics on $Y$, smooth up to
$x=0$ $($i.e., $g$ is the metric of a warped cone$)$. Then the Laplacian
associated with $g$ is a second order cone differential operator on
${\mathbb{D}}$. See Section~\ref{sec:Laplacian} for further details.

A differential boundary condition for $A$ as above is a vector
\begin{equation}
\label{eq:bc}
T=(T_{0},\ldots ,T_{\mu -1}),\qquad T_{j}=\gamma _{0}\circ B_{j}:{
\mathscr{C}}^{\infty ,\infty }({\mathbb{D}},E)\longrightarrow {
\mathscr{C}}^{\infty ,\infty }({\mathbb{B}},F_{j}),
\end{equation}
where each $B_{j}$ is a cone differential operator of order $j$ on
${\mathbb{D}}$ acting from sections of $E$ to sections of some other bundle
$F_{j}^{\prime }$ and where $\gamma _{0}$ denotes the operator of restriction
to the boundary (and $F_{j}$ is the restriction of $F^{\prime }_{j}$ to the
boundary). We allow some of the $F_{j}^{\prime }$ to be of dimension
$0$; in that case the condition $T_{j}$ is void. Setting
$F:= F_{0}\oplus \ldots \oplus F_{\mu -1}$, we will consider $T$ as a map
${\mathscr{C}}^{\infty ,\infty }({\mathbb{D}},E) \to {\mathscr{C}}^{
\infty ,\infty }({\mathbb{B}},F)$.

Given $A$ and $T$ we study the operator $A_{T}$, acting like $A$ on the
domain
\begin{equation}
\label{eq:realization}
{\mathscr{D}}(A_{T})={\mathscr{C}}^{\infty ,\infty }({\mathbb{D}},E)_{T} :=
{\mathscr{C}}^{\infty ,\infty }({\mathbb{D}},E)\cap \ker T
\end{equation}
as an unbounded operator in weighted $L_{p}$-Sobolev spaces
${\mathcal{H}}^{s,\gamma }_{p}({\mathbb{D}},E)$; here $s$ measures smoothness,
on the cylindrical part with respect to $x\partial _{x}$- and
$\partial _{y}$-derivatives, while $\gamma \in {\mathbb{R}}$ refers to a
weight function which coincides with $x^{\gamma }$ on the cylindrical part;
see the appendix for details. The main objective of this article is to
establish the existence of a bounded $H_{\infty }$-calculus for closed extensions
$\underline{A}_{T}$ of $A_{T}$ with domain
${\mathscr{D}}(\underline{A}_{T})\subseteq {\mathcal{H}}^{\mu ,\gamma }_{p}({
\mathbb{D}},E)\cap \ker T$; such extensions are also called
\textit{realizations} of $A$ \textit{subject to the condition $T$}.

This problem has already been considered in \cite{CSS0}, where it has been
shown that a bounded $H_{\infty }$-calculus exists, provided the resolvent
of $\underline{A}_{T}$ has a specific pseudodifferential structure. So far,
however, only few cases were known where the resolvent is of this kind.
Combining the techniques developed in \cite{SS2} for conic manifolds without
boundary with results of Krainer \cite{Kr}, we are now able to treat all
realizations $\underline{A}_{T}$ that are parameter-elliptic in the sense
of Section~\ref{sec:hinfty}. The resolvent is then constructed with the
help of a pseudodifferential calculus for boundary value problems on manifolds
with edges, as presented e.g.~in Kapanadze, Schulze \cite{KS}.

Our methods pertain, in particular, to the Dirichlet and Neumann Laplacian
as discussed in Section~\ref{sec:Laplacian}. As an application we show
in Section~\ref{sec:PME} the existence of a short time solution to the
porous medium equation on the conic manifold ${\mathbb{D}}$ with Neumann
boundary conditions for positive data.

In the appendix of this paper we recall basic definitions of function spaces
on manifolds with conic singularities and present key elements of a calculus
for pseudodifferential operators on manifolds with conical singularities
on the boundary which we will need in the proof of our main theorem.

\section{Boundary value problems for cone differential operators}
\label{sec:bvp}

Let $A$ and $T$ be as in \reftext{\eqref{eq:A}} and \reftext{\eqref{eq:bc}}, respectively.
Consider the boundary value problem
\begin{equation}
\label{eq:fullbvp}
{\mathcal{A}}:=\binom{A}{T}: {\mathscr{C}}^{\infty ,\infty }({\mathbb{D}},E)
\longrightarrow
\begin{matrix}
{\mathscr{C}}^{\infty ,\infty }({\mathbb{D}},E)
\\
\oplus
\\
{\mathscr{C}}^{\infty ,\infty }({\mathbb{B}},F)%
\end{matrix}
.
\end{equation}
After a normalization of the orders of the boundary operators,
${\mathcal{A}}$ can be considered as an element of Boutet de Monvel's algebra
for boundary value problems in the sense of Schrohe, Schulze
\cite{SS_1,SS_2}. In this class, Shapiro-Lopatinskii ellipticity is characterized
by a number of $($principal$)$ symbols associated with each element. This
leads us to associate with ${\mathcal{A}}$ the following symbols:
\begin{enumerate}
\item[(1)] 
The principal symbol
$\pmb{\sigma }_{\psi }^{\mu }({\mathcal{A}})=\pmb{\sigma }_{\psi }^{\mu }(A)$,
an endomorphism of $\pi _{\mathbb{D}}^{*}E$, where
$\pi _{\mathbb{D}}: T^{*}\mathbb{D}_{\textrm{reg}}\setminus 0 \to
\mathbb{D}_{\textrm{reg}}$ is the canonical projection.
\item[(2)] The principal boundary symbol
$\pmb{\sigma }^{\mu }_{\partial }({\mathcal{A}})$, a morphism
\begin{equation*}
{\mathscr{S}}({\mathbb{R}}_{+})\otimes \pi _{\mathbb{B}}^{*}E|_{{
\mathbb{B}}} \longrightarrow {\mathscr{S}}({\mathbb{R}}_{+})\otimes \pi _{
\mathbb{B}}^{*}(E|_{{\mathbb{B}}}\oplus F),
\end{equation*}
where
$\pi _{\mathbb{B}}: T^{*}\mathbb{B}_{\textrm{reg}}\setminus 0 \to
\mathbb{B}_{\textrm{reg}}$ is the canonical projection.
\item[(3)] The conormal symbol $\pmb{\sigma }_{M}^{\mu }({\mathcal{A}})$, a polynomial
on ${\mathbb{C}}$, taking values in boundary problems on $Y$. Its definition
is recalled below.
\end{enumerate}

Both principal symbol and principal boundary symbol degenerate near
$x=0$. Using variables $(x,y,\xi ,\eta )$ on the cotangent bundle of the
cylindrical part $(0,1)\times Y$, the rescaled principal symbol is
\begin{align}
\label{eq:rescaledprinc}
\widetilde{\pmb{\sigma }}_{\psi }^{\mu }({\mathcal{A}})(y,\eta ,\xi )=
\lim _{x\to 0^{+}} x^{\mu }\pmb{\sigma }_{\psi }^{\mu }({\mathcal{A}})(x,y,x^{-1}
\xi ,\eta ).
\end{align}
Similarly, with $(x,y',\xi ,\eta ')$ in the cotangent bundle of
$(0,1)\times \partial Y$, the rescaled principal boundary symbol is
\begin{align}
\label{rescaledbdry}
\widetilde{\pmb{\sigma }}^{\mu }_{\partial }({\mathcal{A}})(y',\xi ,
\eta ') =\lim _{x\to 0^{+}}
\binom{x^{\mu }\pmb{\sigma }_{\partial }^{\mu }(A)(x,y',x^{-1}\xi ,\eta ')}{S(x)\pmb{\sigma }^{\mu }_{\partial }(T)(x,y',x^{-1}\xi ,\eta ')},
\end{align}
where $S(x)=\mathrm{diag}(1,x,\ldots , x^{\mu -1})$. The rescaled principal
symbol is an endomorphism of $\pi ^{*}E_{0}$, while the rescaled principal
boundary symbol is a morphism
${\mathscr{S}}({\mathbb{R}}_{+})\otimes \pi ^{*}_{\partial Y} E_{0}|_{
\partial Y} \to {\mathscr{S}}({\mathbb{R}}_{+})\otimes \pi ^{*}_{
\partial Y}(E_{0}|_{\partial Y}\oplus F_{0})$.

The conormal symbol of $A$ is the polynomial
\begin{align}
\label{eq:conormalA}
\pmb{\sigma }_{M}^{\mu }(A)(z) = \sum _{j=0}^{\mu }a_{j}(0) z^{j}: {
\mathbb{C}}\longrightarrow \mathrm{Diff}^{\mu }(Y,E_{0})
\end{align}
with $a_{j}$ as in \reftext{\eqref{eq:A}}, taking values in the differential operators
of order $\mu -j$ on the cross-section $Y$. For the boundary condition
$T$
\begin{align}
\label{eq:conormalT}
\pmb{\sigma }_{M}^{\mu }(T)(z)=\Big (\gamma _{0} \pmb{\sigma }_{M}^{0}(B_{0})(z),
\ldots ,\gamma _{0} \pmb{\sigma }_{M}^{\mu -1}(B_{\mu -1})(z)\Big ),
\end{align}
where $\gamma _{0}$ denotes the operator of restriction to the boundary
of $Y$. In particular, the conormal symbol
$\pmb{\sigma }^{\mu }_{M}({\mathcal{A}})=
\binom{\pmb{\sigma }^{\mu }_{M}(A)}{\pmb{\sigma }^{\mu }_{M}(T)}$ furnishes
a map
\begin{align}
\label{eq:conormal-symbol}
\pmb{\sigma }_{M}^{\mu }({\mathcal{A}})(z): H^{s}_{p}(Y,E_{0})
\longrightarrow
\begin{matrix}
H^{s-\mu }_{p}(Y,E_{0})
\\
\oplus
\\
\mathop{\text{$\oplus $}}\limits _{j=0}^{\mu -1} H^{s-j-1/p}_{pp}(
\partial Y,F_{0})
\end{matrix},\qquad z\in {\mathbb{C}},
\end{align}
for every $1<p<+\infty $ and $s>\mu -1+1/p$ $($with a slight abuse of notation,
the trace spaces on the right are the usual Besov spaces with indices
$p=q)$.

\begin{definition}%
\label{def:elliptic}
The boundary value problem ${\mathcal{A}}$ is called \emph{$\mathbb{D}$-elliptic}, if each of the four symbols
$\pmb{\sigma }_{\psi }^{\mu }({\mathcal{A}})$,
$\widetilde{\pmb{\sigma }}_{\psi }^{\mu }({\mathcal{A}})$,
$\pmb{\sigma }_{\partial }^{\mu }({\mathcal{A}})$, and
$\widetilde{\pmb{\sigma }}_{\partial }^{\mu }({\mathcal{A}})$ is invertible
outside the zero-section.

It is called elliptic with respect to the weight
$\gamma \in {\mathbb{R}}$, if additionally the conormal symbol \reftext{\eqref{eq:conormal-symbol}} is invertible for all $z\in {\mathbb{C}}$ with
$\revtex z = \frac{n+1}{2}-\gamma $.
\end{definition}

It can be shown, cf. \cite[Theorem 4.1.6]{SS_2}, that the conormal symbol
of a ${\mathbb{D}}$-elliptic boundary value problem is meromorphically invertible
in the complex plane, independently of the choice of $s$ and $p$, due to
spectral invariance in Boutet de Monvel's algebra of boundary value problems.
Ellipticity with respect to a weight $\gamma $ thus is just the requirement
that none of the poles has real part equal to $\frac{n+1}{2}-\gamma $.

\subsection{${\mathbb{D}}$-ellipticity with parameter}
\label{sec2.1}

Of crucial importance for this paper will be the notion of ellipticity
with respect to a parameter in a sector. For $\theta >0$ define
\begin{equation}
\label{Lambda}
\Lambda = \Lambda _{\theta }= \{re^{i\varphi }\in {\mathbb{C}}\mid r\ge 0,
\ \theta \leq \varphi \leq 2\pi -\theta \}.
\end{equation}

\begin{definition}%
\label{def:parameter-elliptic}
The boundary value problem ${\mathcal{A}}$ is called ${\mathbb{D}}$-elliptic
with parameter in $\Lambda $, if
$\lambda -\pmb{\sigma }_{\psi }^{\mu }({\mathcal{A}})$,
$\lambda -\widetilde{\pmb{\sigma }}_{\psi }^{\mu }({\mathcal{A}})$,
$\binom{\lambda }{0}-\pmb{\sigma }_{\partial }^{\mu }({\mathcal{A}})$ and
$\binom{\lambda }{0}-\widetilde{\pmb{\sigma }}_{\partial }^{\mu }({
\mathcal{A}})$ are invertible for every $\lambda \in \Lambda $.
\end{definition}

Ellipticity with parameter implies that $T$ is a normal boundary condition
in the sense of Grubb \cite[Definition 1.4.3]{Gr} and
\cite[Definition 3.6]{CSS1}; the argument is the same as in the smooth
case, see \cite[Lemma 1.5.7]{Gr}. In particular
\begin{equation}
\label{eq:Tsurjective}
T:{\mathcal{H}}^{s,\gamma }_{p}({\mathbb{D}},E) \longrightarrow \mathop{
\text{$\oplus $}}\limits _{j=0}^{\mu -1} {\mathcal{H}}^{s-j-1/p,\gamma -j-1/2}_{pp}({
\mathbb{B}},F_{j})
\end{equation}
is surjective for every choice of $\gamma \in {\mathbb{R}}$,
$1<p<+\infty $ and $s>\mu -1+1/p$. There exists a right-inverse in Boutet
de Monvel's calculus, as constructed for example in Lemma 3.4 and Proposition
3.7 of \cite{CSS1}. The trace spaces on the right are weighted Besov spaces
with parameters $p=q$. Though some of the discussion below remains valid
for ${\mathbb{D}}$-elliptic boundary value problems with normal boundary
condition, we shall make the following assumption which stands throughout
the whole paper:

\medskip\noindent
\textbf{Assumption:} The boundary value problem
${\mathcal{A}}=\binom{A}{T}$ is ${\mathbb{D}}$-elliptic with parameter in
$\Lambda $.

\section{Domains and realizations}
\label{sec3}

In this section let $\gamma \in {\mathbb{R}}$ and $1<p<+\infty $ be fixed.

\subsection{Closed extensions of the full boundary value problem}
\label{sec:binomAT}

Let us now consider ${\mathcal{A}}$ as an unbounded operator
\begin{align}
\label{eq:sA-unbounded}
{\mathcal{A}}:{\mathscr{C}}^{\infty ,\infty }({\mathbb{D}},E)\subset {
\mathcal{H}}^{0,\gamma }_{p}({\mathbb{D}},E) \longrightarrow
\begin{matrix}
{\mathcal{H}}^{0,\gamma }_{p}({\mathbb{D}},E)
\\
\oplus
\\
\mathop{\text{$\oplus $}}\limits _{j=0}^{\mu -1} {\mathcal{H}}^{\mu -j-1/p,
\gamma +\mu -j-1/2}_{pp}({\mathbb{B}},F_{j}).
\end{matrix}
\end{align}
We let ${\mathscr{D}}_{\min }({\mathcal{A}})$ denote the domain of the closure,
while ${\mathscr{D}}_{\max }({\mathcal{A}})$ consists of all elements
$u\in {\mathcal{H}}^{\mu ,\gamma }_{p}({\mathbb{D}},E)$ such that
${\mathcal{A}}u$ belongs to the space on the right-hand side of \reftext{\eqref{eq:sA-unbounded}}. Note that in the definition of
${\mathscr{D}}_{\max }({\mathcal{A}})$ the a-priori regularity $\mu $ is required.

Obviously it would be more precise to use notations like
${\mathscr{D}}_{\max }^{\gamma ,p}({\mathcal{A}})$ to indicate the dependence
on $\gamma $ and $p$. However, to keep notation more lean we shall not
do so.

The first statement, below, is shown in \cite[Lemma 4.10]{Kr}, the second
follows from \cite[Proposition 4.3]{CSS1}.

\begin{theorem}%
\label{thm:Dmin}
The minimal domain satisfies
\begin{equation*}
{\mathscr{D}}_{\min }({\mathcal{A}})={\mathscr{D}}_{\max }({\mathcal{A}})
\cap \mathop{\text{\large $\cap $}}_{\varepsilon >0}{\mathcal{H}}^{\mu ,
\gamma +\mu -\varepsilon }_{p}({\mathbb{D}},E).
\end{equation*}
If ${\mathcal{A}}$ is elliptic with respect to the weight
$\gamma +\mu $, then
\begin{equation*}
{\mathscr{D}}_{\min }({\mathcal{A}})={\mathcal{H}}^{\mu ,\gamma +\mu }_{p}({
\mathbb{D}},E).
\end{equation*}
\end{theorem}

In the following theorem we write $Y^{\wedge }=(0,+\infty )\times Y$ and
let $E^{\wedge }$ denote the pull-back of $E_{0}=E_{|x=0}$ under the canonical
projection $Y^{\wedge }\to Y$. We call
$\omega \in {\mathscr{C}}^{\infty }_{c}([0,1))$ a cut-off function, if it
is non-negative and $\equiv 1 $ near $0$. We will also consider
$\omega $ as a function on ${\mathbb{D}}$, supported in the cylindrical
part of ${\mathbb{D}}$.

\begin{theorem}%
\label{thm:Dmax}
There exists a finite-dimensional space
${\mathscr{F}}\subset {\mathscr{C}}^{\infty }(Y^{\wedge },E^{\wedge })$ such that
$\omega {\mathscr{F}}\subset {\mathcal{H}}^{\infty ,\gamma +
\varepsilon }_{p}({\mathbb{D}})$ for some $\varepsilon >0$ and
\begin{equation*}
{\mathscr{D}}_{\max }({\mathcal{A}})={\mathscr{D}}_{\min }({\mathcal{A}})
\oplus \omega {\mathscr{F}},
\end{equation*}
where $\omega $ is an arbitrary cut-off function. In particular, any closed
extension $\underline{{\mathcal{A}}}\subset {\mathcal{A}}_{\max }$ is given
by a domain of the form
\begin{equation*}
{\mathscr{D}}(\underline{{\mathcal{A}}})={\mathscr{D}}_{\min }({
\mathcal{A}})\oplus \omega \underline{{\mathscr{F}}},\qquad
\text{$\underline{{\mathscr{F}}}$ subspace of ${\mathscr{F}}$}.
\end{equation*}
\end{theorem}

Using the identification
\begin{equation*}
{\mathscr{C}}^{\infty }(Y^{\wedge },E^{\wedge })={\mathscr{C}}^{\infty }\big ((0,+
\infty ),{\mathscr{C}}^{\infty }(Y,E_{0})\big ),
\end{equation*}
the elements of ${\mathscr{F}}$ are linear combinations of functions of
the form
\begin{align}
\label{eq:F}
u(x,y)=v(y) x^{-q}\log ^{k} x,\quad v\in {\mathscr{C}}^{\infty }(Y,E_{0}),
\end{align}
with complex numbers $q$ satisfying
$\frac{n+1}{2}-\gamma -\mu \le \revtex q<\frac{n+1}{2}-\gamma $, and integers
$k\ge 0$. A detailed description can be found in \cite[Section 6]{Kr}.
We shall present here an alternative, which combines the corresponding
result for manifolds without boundary from \cite[Section 3]{SS2} with the
approach of \cite[Sections 4 and 5]{CSS1}.

Similarly to \reftext{\eqref{eq:conormalA}} and \reftext{\eqref{eq:conormalT}} we shall introduce
so-called lower order conormal symbols as
\begin{align}
\label{eq:conormalA_low}
\pmb{\sigma }_{M}^{\mu -k}(A)(z) = \frac{1}{k!}\sum _{j=0}^{\mu
}\frac{d^{k}a_{j}}{dx^{k}}(0) z^{j}, \qquad k\in {\mathbb{N}},
\end{align}
and
\begin{align}
\label{eq:conormalT_low}
\pmb{\sigma }_{M}^{\mu -k}(T)(z)=\Big (\gamma _{0} \pmb{\sigma }_{M}^{-k}(B_{0})(z),
\ldots ,\gamma _{0} \pmb{\sigma }_{M}^{\mu -1-k}(B_{\mu -1})(z)\Big ).
\end{align}
We write
\begin{equation*}
\mathbf{f}_{k}:=\pmb{\sigma }_{M}^{\mu -k}({\mathcal{A}})=
\binom{\pmb{\sigma }_{M}^{\mu -k}(A)}{\pmb{\sigma }_{M}^{\mu -k}(T)},
\qquad k\in {\mathbb{N}}_{0},
\end{equation*}
and define recursively the following functions:
\begin{equation}
\label{eq:recursion}
\mathbf{g}_{0}=1,\qquad \mathbf{g}_{\ell }=-(\mathrm{T}^{-\ell }
\mathbf{f}^{-1}_{0})\sum _{j=0}^{\ell -1}(\mathrm{T}^{-j}
\mathbf{f}_{\ell -j})\mathbf{g}_{j},\qquad \ell \in {\mathbb{N}},
\end{equation}
where the shift-operator $\mathrm{T}^{\rho }$, $\rho \in {\mathbb{R}}$, acts
on meromorphic functions by $(\mathrm{T}^{\rho }h)(z)=h(z+\rho )$. Note that
the $\mathbf{g}_{\ell }$ are meromorphic with values in Boutet de Monvel's
algebra on $Y$ and that the recursion is equivalent to
\begin{equation}
\label{eq:recursion2}
\sum _{j=0}^{\ell }(\mathrm{T}^{-j}\mathbf{f}_{\ell -j})\mathbf{g}_{j}
=
\begin{cases}
\mathbf{f}_{0}&\quad : \ell =0
\\
0&\quad : \ell \ge 1%
\end{cases}
.
\end{equation}
In the following we let
\begin{equation}
\label{eq:strip}
S_{\gamma }=\Big \{\sigma \in {\mathbb{C}}\mid \sigma
\text{ is a pole of $\mathbf{f}_{0}^{-1}$ and } \frac{n+1}{2}-\gamma -
\mu <\revtex \sigma <\frac{n+1}{2}-\gamma \Big \}
\end{equation}
and we use the Mellin transform
\begin{equation*}
({\mathcal{M}}u)(z) = \widehat{u}(z)=\int _{0}^{\infty }x^{z} u(x)
\frac{dx}{x}.
\end{equation*}

\begin{theorem}%
\label{thm:theta}
For $\sigma \in S_{\gamma }$ and $\ell \in {\mathbb{N}}_{0}$ define
\begin{equation*}
\mathbf{G}_{\sigma }^{(\ell )}:
\begin{matrix}
{\mathscr{S}}^{\infty }(Y^{\wedge },E^{\wedge })
\\
\oplus
\\
{\mathscr{S}}^{\infty }(\partial Y^{\wedge },F^{\wedge })
\end{matrix}
\longrightarrow {\mathscr{C}}^{\infty }(Y^{\wedge },E^{\wedge })
\end{equation*}
by
\begin{equation}
\label{eq:gsigmal}
(\mathbf{G}_{\sigma }^{(\ell )} u)(x)= x^{\ell }\,\int _{|z-\sigma |=
\varepsilon } x^{-z}\mathbf{g}_{\ell }(z)\,\Pi _{\sigma }(\mathbf{f}_{0}^{-1}
\,\widehat{u})(z)\,\dbar z,
\end{equation}
where $\Pi _{\sigma }h$ denotes the principal part of the Laurent series
of a meromorphic function $h$ in $\sigma $ $(\Pi _{\sigma }h=0$, if
$h$ is holomorphic in $\sigma )$, $\dbar z= \frac{1}{2\pi i} dz$, and
$\varepsilon >0$ is so small that none of the $z$ with
$0<|z-\sigma |<\varepsilon $ is a pole of the integrand. Moreover, let
\begin{equation}
\label{eq:gsigma}
\mathbf{G}_{\sigma }:=\sum _{\ell =0}^{\mu _{\sigma }}\mathbf{G}_{\sigma }^{(
\ell )},\qquad \mu _{\sigma }:=\Big \lfloor \revtex \sigma +\mu +\gamma -
\frac{n+1}{2}\Big \rfloor,
\end{equation}
where $\lfloor r\rfloor$ denotes the integer part of $r\in {\mathbb{R}}$. Then
\begin{equation*}
{\mathscr{F}}=\mathop{\text{\Large $\oplus $}}_{\sigma \in S_{\gamma }} {
\mathscr{F}}_{\sigma }, \qquad {\mathscr{F}}_{\sigma }:=\mathrm{range}\,
\mathbf{G}_{\sigma }.
\end{equation*}
\end{theorem}

The proof is a straightforward adaption of the proof of
\cite[Section 3]{SS2}.

\subsection{Realizations subject to a boundary condition}
\label{3.2}

We shall next determine the closed extensions in
${\mathcal{H}}^{0,\gamma }_{p}({\mathbb{D}})$ of the operator $A_{T}$ introduced
in \reftext{\eqref{eq:realization}}. Let ${\mathscr{D}}_{\min }(A_{T})$ be the domain
of the closure of $A_{T}$ and define the maximal extension of
$A_{T}$ by the action of $A$ on
\begin{equation*}
{\mathscr{D}}_{\max }({A}_{T})=\big \{u\in \mathcal{H}^{\mu ,\gamma }_{p}(
\mathbb{D},E)\mid Au\in \mathcal{H}^{0,\gamma }_{p}(\mathbb{D,E}),\;Tu=0
\big \}.
\end{equation*}

There is a natural relation between the closed extensions of
${\mathcal{A}}$ and those of $A_{T}$:

\begin{lemma}[\text{\cite[Lemma 4.10]{Kr}}]
\label{ce}
We have
\begin{equation}
\label{eq:quotient-isom}
{\mathscr{D}}_{\max }({\mathcal{A}})/{\mathscr{D}}_{\min }({\mathcal{A}})
\cong {\mathscr{D}}_{\max }(A_{T})/{\mathscr{D}}_{\min }(A_{T})
\end{equation}
and the map
\begin{equation*}
V\mapsto V_{T}:=V\cap \mathrm{ker}\,T
\end{equation*}
is a bijection between the lattice of intermediate spaces
${\mathscr{D}}_{\min }({\mathcal{A}})\subseteq V\subseteq {\mathscr{D}}_{
\max }({\mathcal{A}})$ and the lattice of intermediate spaces
${\mathscr{D}}_{\min }(A_{T})\subseteq V_{T}\subseteq {\mathscr{D}}_{
\max }(A_{T})$.
\end{lemma}

More precisely, let $R$ be a right-inverse of $T$ as in \reftext{\eqref{eq:Tsurjective}} with $\gamma $ replaced by $\gamma +\mu $. Then
\begin{equation}
\label{eq:quotient-isom2}
u+{\mathscr{D}}_{\min }({\mathcal{A}})\mapsto (u-RTu)+{\mathscr{D}}_{\min }(A_{T})
\end{equation}
provides an isomorphism \reftext{\eqref{eq:quotient-isom}}. Note that $1-RT$ is a
projection onto the kernel of $T$, since $T(1-RT)=0$. In particular, if
${\mathscr{F}}$ is the space from \reftext{Theorem~\ref{thm:Dmax}}, then
\begin{equation}
\label{eq:complement}
{\mathscr{D}}_{\max }(A_{T})={\mathscr{D}}_{\min }(A_{T})\oplus {
\mathscr{E}},\qquad {\mathscr{E}}:=\{\omega u-RT(\omega u)\mid u\in {
\mathscr{F}}\},
\end{equation}
provides a (non-canonical) description of the maximal domain. We obtain:

\begin{theorem}
Any extension $\underline{A}_{T}$ of $A_{T}$ with
$\underline{A}_{T}\subseteq A_{T,\max }$ corresponds to a choice of a subspace
of ${\mathscr{E}}$, i.e. has a domain of the form
\begin{equation}
\label{E}
\mathscr{D}(\underline{A}_{T})=\mathscr{D}_{\min }({A}_{T})\oplus
\underline{{\mathscr{E}}}, \qquad \underline{{\mathscr{E}}}
\text{ subspace of } {\mathscr{E}}.
\end{equation}
\end{theorem}

By \reftext{\eqref{eq:F}},
${\mathscr{E}}\subset {\mathcal{H}}^{\infty ,\gamma +\varepsilon }_{p}({
\mathbb{D}},E)$ for some $\varepsilon >0$. Moreover, since
$T(\omega u)$ vanishes for small $x$,
$RT(\omega {\mathscr{F}})\subseteq {\mathcal{H}}^{\infty ,\gamma +\mu }_{p}({
\mathbb{D}},E)$. In particular,
\begin{equation*}
{\mathscr{E}}+{\mathcal{H}}^{\infty ,\gamma +\mu }_{p}({\mathbb{D}},E)=
\omega {\mathscr{F}}+{\mathcal{H}}^{\infty ,\gamma +\mu }_{p}({\mathbb{D}},E).
\end{equation*}

If ${\mathcal{A}}$ is elliptic with respect to the weight
$\gamma +\mu $, then, by \cite[Proposition 4.3]{CSS1},
\begin{equation*}
{\mathscr{D}}_{\min }(A_{T}) = {\mathcal{H}}^{\mu ,\gamma +\mu }_{p}({
\mathbb{D}})_{T}:={\mathcal{H}}^{\mu ,\gamma +\mu }_{p}({\mathbb{D}})
\cap \mathrm{ker}\,T.
\end{equation*}

\subsection{Operators on the model cone}
\label{sec3.3}

If $A$ is as in \reftext{\eqref{eq:A}} we define the differential operator
\begin{equation*}
\widehat{A}=x^{-\mu }\sum _{j=0}^{\mu }a_{j}(0)(-x\partial _{x})^{j}
\end{equation*}
by freezing the coefficients of $A$ in $x=0$. Proceeding analogously with
the operators $B_{k}$, we obtain the boundary condition
$\widehat{T}=(\gamma _{0}\widehat{B}_{0},\ldots ,\gamma _{0}
\widehat{B}_{\mu -1})$. Together they define the so-called model boundary
value problem
$\widehat{\mathcal{A}}=\binom{\widehat{A}}{\widehat{T}}$ on the infinite
cylinder $Y^{\wedge }=(0,+\infty )\times Y$. Similarly as above in Section~\ref{sec:binomAT}, we shall consider $\widehat{\mathcal{A}}$ as an unbounded
operator
\begin{align}
\label{eq:whsA-unbounded}
\widehat{\mathcal{A}}: {\mathscr{S}}^{\infty }(Y^{\wedge },E^{\wedge })
\subset {\mathcal{K}}^{0,\gamma }_{p}(Y^{\wedge },E^{\wedge })
\longrightarrow
\begin{matrix}
{\mathcal{K}}^{0,\gamma }_{p}(Y^{\wedge },E^{\wedge })
\\
\oplus
\\
\mathop{\text{$\oplus $}}\limits _{k=0}^{\mu -1} {{\mathcal{K}}}^{\mu -k-1/p,
\gamma +\mu -k-1/2}_{pp}(\partial Y^{\wedge },{F}_{k}^{\wedge })
\end{matrix};
\end{align}
again, the trace spaces on the right-hand side are Besov spaces with
$p=q$.

The discussion of the corresponding closed extensions is parallel to that
above. The domain of the closure
${\mathscr{D}}_{\min }(\widehat{\mathcal{A}})$ and the maximal domain
${\mathscr{D}}_{\max }(\widehat{\mathcal{A}})$ are contained in
${\mathcal{K}}^{\mu ,\gamma }_{p}(Y^{\wedge },E^{\wedge })$ and
\begin{equation*}
{\mathscr{D}}_{\min }(\widehat{\mathcal{A}})={\mathscr{D}}_{\max }(
\widehat{\mathcal{A}})\cap \mathop{\text{\large $\cap $}}_{
\varepsilon >0}{\mathcal{K}}^{\mu ,\gamma +\mu -\varepsilon }_{p}(Y^{\wedge },E^{\wedge }).
\end{equation*}
If ${\mathcal{A}}$ is elliptic with respect to the weight
$\gamma +\mu $, then
\begin{equation}
\label{eq:Dmin}
{\mathscr{D}}_{\min }(\widehat{\mathcal{A}})={\mathcal{K}}^{\mu ,\gamma +
\mu }_{p}(Y^{\wedge },E^{\wedge }).
\end{equation}

The following theorem combines the results of Krainer in
\cite[Section 6]{Kr} with the above representation.
\begin{theorem}%
\label{thm:Dmaxhat}
Let $S_{\gamma }$ be as in \reftext{\eqref{eq:strip}} and
$\mathbf{G}^{(\ell )}_{\sigma }$, $\mathbf{G}_{\sigma }$ be as in \reftext{\eqref{eq:gsigmal}} and \reftext{\eqref{eq:gsigma}}, respectively. Then
\begin{equation*}
{\mathscr{D}}_{\max }(\widehat{\mathcal{A}})={\mathscr{D}}_{\min }(
\widehat{\mathcal{A}})\oplus \omega \widehat{\mathscr{F}},
\end{equation*}
where
\begin{equation*}
\widehat{\mathscr{F}}=\mathop{\text{\Large $\oplus $}}_{\sigma \in S_{\gamma }} \widehat{\mathscr{F}}_{\sigma }, \qquad \widehat{\mathscr{F}}_{\sigma }:=\mathrm{range}\,\mathbf{G}_{\sigma }^{(0)}.
\end{equation*}
The mappings
\begin{equation*}
\theta _{\sigma }:{\mathscr{F}}_{\sigma }\longrightarrow
\widehat{\mathscr{F}}_{\sigma },\quad \mathbf{G}_{\sigma }u\mapsto
\mathbf{G}_{\sigma }^{(0)}u
\end{equation*}
are well-defined isomorphisms, hence induce an isomorphism
$\theta :{\mathscr{F}}\to \widehat{\mathscr{F}}$. In particular,
\begin{equation}
\label{eq:quotient-isom-hat}
{\mathscr{D}}_{\max }({\mathcal{A}})/{\mathscr{D}}_{\min }({\mathcal{A}})
\cong {\mathscr{D}}_{\max }(\widehat{\mathcal{A}})/{\mathscr{D}}_{\min }(
\widehat{\mathcal{A}}).
\end{equation}
\end{theorem}

In other words, the map $\theta $ gives rise to a bijection
$\Theta $ between the subspaces of ${\mathscr{F}}$ and those of
$\widehat{\mathscr{F}}$, i.e. an isomorphism
\begin{equation}
\label{eq:Grassmannian}
\Theta :\mathrm{Gr}({\mathscr{F}})\longrightarrow \mathrm{Gr}(
\widehat{\mathscr{F}})
\end{equation}
between the corresponding Grassmannians. This isomorphism has been first
described in \cite{GKM1} in the case of manifolds without boundary and
in \cite{Kr} in the case with boundary; the present, equivalent, construction
extends that of \cite{SS2} to the case with boundary. It induces a one-to-one
correspondence between the closed extensions of ${\mathcal{A}}$ and those
of $\widehat{\mathcal{A}}$ by
\begin{equation}
\label{eq:theta_dom}
{\mathscr{D}}(\underline{{\mathcal{A}}})={\mathscr{D}}_{\min }({
\mathcal{A}})\oplus \omega \underline{{\mathscr{F}}}\mapsto {\mathscr{D}}(
\underline{\widehat{\mathcal{A}}}):={\mathscr{D}}_{\min }(
\widehat{\mathcal{A}})\oplus \omega \Theta \underline{{\mathscr{F}}}.
\end{equation}

\begin{remark}
By construction,
$\widehat{\mathscr{F}}\subseteq \ker \widehat{\mathcal{A}}$. In fact,
\begin{align*}
\widehat{\mathcal{A}}\mathbf{G}_{\sigma }^{(0)}u(x) &=
\widehat{\mathcal{A}}\int _{|z-\sigma |=\varepsilon } x^{-z}\Pi _{\sigma }(\mathbf{f}_{0}^{-1}\,\widehat{u})(z)\,\dbar z
\\
&= \int _{|z-\sigma |=\varepsilon } x^{-z} \mathbf{f}_{0}(z) (
\mathbf{f}_{0}^{-1}\,\widehat{u})(z)\,\dbar z = 0,
\end{align*}
since $\Pi _{\sigma }$ can be omitted by the residue theorem. In particular,
$\widehat{\mathscr{F}}\subseteq \ker \widehat{T}$.
\end{remark}

Our next topic are the realizations of $\widehat{A}$ in
$\mathcal{K}^{0,\gamma }_{p}(Y^{\wedge },E^{\wedge })$ subject to the boundary
condition $\widehat{T}$, i.e., the extensions of the unbounded operator
$\widehat{A}_{\widehat{T}}$ acting like $\widehat{A}$ on the domain
\begin{align}
\label{wAT}
\{u\in {\mathscr{S}}^{\infty }(Y^{\wedge },E^{\wedge })\mid \widehat{T}u=0\}
\subseteq \mathcal{K}^{0,\gamma }_{p}(Y^{\wedge },E^{\wedge }).
\end{align}
Let ${\mathscr{D}}_{\min }(\widehat{A}_{\widehat{T}})$ be the domain of the
closure of $\widehat{A}_{\widehat{T}}$, and define the maximal extension
${\widehat{A}}_{\widehat{T},\max }$ by the action of $\widehat{A}$ on
\begin{equation*}
{\mathscr{D}}_{\max }({\widehat{A}}_{\widehat{T}})=\big \{u\in \mathcal{K}^{
\mu ,\gamma }_{p}(Y^{\wedge },E^{\wedge })\mid Au\in \mathcal{K}^{0,\gamma }_{p}(Y^{\wedge },E^{\wedge }),\;\widehat{T}u=0\big \}.
\end{equation*}
Proceeding as above, using a right-inverse $\widehat{R}$ to
$\widehat{T}$ in Boutet de Monvels calculus on the infinite cone, we find
a non-canonical decomposition
\begin{equation}
\label{eq:complement-hat}
{\mathscr{D}}_{\max }(\widehat{A}_{\widehat{T}})={\mathscr{D}}_{\min }(
\widehat{A}_{\widehat{T}})\oplus \widehat{\mathscr{E}},\qquad
\widehat{\mathscr{E}}:=\{\omega u-\widehat{R}\widehat{T}(\omega u)\mid u
\in \widehat{\mathscr{F}}\}.
\end{equation}
Since both the differential boundary condition $\widehat{T}$ and the right-inverse
$\widehat{R}$ preserve rapid decay of functions at infinity, there exists
an $\varepsilon >0$ such that
\begin{equation}
\label{eq:rapid-decay}
\widehat{{\mathscr{E}}}\subset {\mathscr{S}}^{\gamma +\varepsilon }(Y^{\wedge },E^{\wedge }).
\end{equation}

\begin{remark}
$\widehat{R}$ can be chosen such that it commutes with dilations, i.e. given
$u\in {\mathcal{K}}^{\mu ,\gamma }_{p}(\partial Y^{\wedge },\break  E^{\wedge })$ and
writing $u_{\lambda }(x,y) = u(\lambda x,y)$ for $\lambda >0$ we have
$\widehat{R}u_{\lambda }= (\widehat{R} u)_{\lambda }$. This follows from the
fact that $\widehat{T}$ commutes with dilations together with the construction
in \cite[Section 3]{CSS1}.
\end{remark}

\begin{remark}%
\label{rem:RT}
In general, multiplication by a cut-off function $\omega $ will not commute
with the trace operator $\widehat{T}$. If $\omega $ and $\widehat{T}$ commute,
$\widehat{T}(\omega u)=0$ for $u\in \widehat{\mathscr{F}}$, so that
$\widehat{\mathscr{E}}= \omega \widehat{\mathscr{F}}$. This is the case
e.g. for Dirichlet and Neumann boundary conditions.
\end{remark}

The isomorphism $\Theta $ provides a one-to-one correspondence between
the closed extensions of $A_{T}$ and $\widehat{A}_{\widehat{T}}$, which will
be of crucial importance below:

\begin{definition}%
\label{def:domain_hat_A}
Let $\underline{A}_{T}$ be a closed extension of $A_{T}$. According to \reftext{\eqref{E}} and \reftext{\eqref{eq:complement}} it has domain
$\mathscr{D}(\underline{A}_{T})=\mathscr{D}_{\min }({A}_{T})\oplus
\underline{{\mathscr{E}}}$, where
\begin{equation*}
\underline{{\mathscr{E}}}=\{\omega u-RT(\omega u)\mid u\in
\underline{{\mathscr{F}}}\},\qquad \underline{{\mathscr{F}}}\subset {
\mathscr{F}}.
\end{equation*}
Then let $\underline{\widehat{A}}_{\widehat{T}}$ be the closed extension
of $\widehat{A}_{\widehat{T}}$ defined by the domain
\begin{equation*}
\mathscr{D}(\underline{\widehat{A}}_{\widehat{T}}) =\mathscr{D}_{\min }({
\widehat{A}}_{\widehat{T}})\oplus \underline{\widehat{\mathscr{E}}},
\qquad \underline{\widehat{\mathscr{E}}}=\{\omega u-\widehat{R}
\widehat{T}(\omega u)\mid u\in \Theta \underline{{\mathscr{F}}}\}.
\end{equation*}
\end{definition}

The mapping
${\mathscr{D}}(\underline{A}_{T})\mapsto {\mathscr{D}}(
\underline{\widehat{A}}_{\widehat{T}})$ does not depend on the choice of
the right-inverses $R$ and $\widehat{R}$, respectively, but only on the
isomorphism $\Theta $ from \reftext{\eqref{eq:Grassmannian}}.

Finally let us remark that, in case of ellipticity with respect to the
weight $\gamma +\mu $,
\begin{align}
\label{eq:Dmin2}
\mathscr{D}_{\min }({\widehat{A}}_{\widehat{T}})&={\mathcal{K}}^{\mu ,
\gamma +\mu }_{p}(Y^{\wedge },E^{\wedge })_{\widehat{T}} :={\mathcal{K}}^{
\mu ,\gamma +\mu }_{p}(Y^{\wedge },E^{\wedge })\cap \mathrm{ker}\,
\widehat{T}.
\end{align}

\subsection{Invertibility and Fredholm property}
\label{sec3.4}

\begin{proposition}%
\label{prop:independent}
Consider the operator
\begin{equation*}
\underline{A}_{T}:{\mathcal{H}}^{s+\mu ,\gamma +\mu }_{p}({\mathbb{D}}, E)_{T}
\oplus \underline{{\mathscr{E}}}\longrightarrow {\mathcal{H}}^{s,
\gamma }_{p}({\mathbb{D}}, E).
\end{equation*}
If this is a Fredholm operator for some choice $s=s_{0}$, $p=p_{0}$ with
$\mu \le s_{0}\in {\mathbb{N}}_{0}$, $1<p_{0}<+\infty $, then it is a Fredholm
operator for all $1<p<+\infty $, $s\in {\mathbb{Z}}$, $s>\mu -1+1/p$. Also
the index then is independent of $s$ and $p$. Similarly, if it is invertible
for $s_{0}$, $p_{0}$, then it is invertible for all $1<p<+\infty $,
$s>\mu -1+1/p$.
\end{proposition}

\begin{proof}
Since $T$ is surjective and $\underline{{\mathscr{E}}}$ is finite-dimensional,
\cite[Theorem 8.3]{CSS1} implies that $\underline{A}_{T}$ is a Fredholm
operator if and only if
\begin{equation*}
{\mathcal{A}}:{\mathcal{H}}^{s+\mu ,\gamma +\mu }_{p}({\mathbb{D}},E)
\longrightarrow
\begin{matrix}
{\mathcal{H}}^{s,\gamma }_{p}({\mathbb{D}},E)
\\
\oplus
\\
\mathop{\text{$\oplus $}}\limits _{j=0}^{\mu -1} {\mathcal{H}}_{pp}^{s+
\mu -j-1/p,\gamma +\mu -j-1/2}({\mathbb{B}},F_{j})
\end{matrix}
\end{equation*}
is a Fredholm operator; in that case, their indices differ by the dimension
of $\underline{{\mathscr{E}}}$. By Corollary 50 in \cite{LS} the Fredholm
property of ${\mathcal{A}}$ implies that it is -- after normalization of
the orders of the boundary operators -- an elliptic cone pseudodifferential
operator in the sense of \cite{LS}. Hence it has a parametrix and therefore
is a Fredholm operator for all other choices of $s$ and $p$. Moreover,
the index is independent of $s$ and $p$ by \cite[Corollary 50]{LS}.

Next suppose $A_{T}$ is invertible for some fixed choice
$s_{0},p_{0}$. It follows from the first part of the proof that
$A_{T}$ is a Fredholm operator of index zero for the other values of
$s$ and $p$. Hence it will suffice to establish the injectivity of
$A_{T}$. Suppose $A(u+e)=0$ for some
$u\in {\mathcal{H}}^{s+\mu ,\gamma +\mu }_{p}({\mathbb{D}})_{T}$ and
$e\in \underline{{\mathscr{E}}}$. Then
$Au=-Ae\in {\mathscr{C}}^{\infty ,\gamma +\varepsilon }({\mathbb{D}})$ for
some $\varepsilon >0$. By elliptic regularity in the cone algebra we conclude
that
$u\in {\mathscr{C}}^{\infty ,\mu +\gamma +\varepsilon '}({\mathbb{D}})$ for
some $\varepsilon '>0$. This shows that the kernel of
$\underline{A}_{T}$ does not depend on $s$ and $p$. Hence $A_{T}$ is invertible
for all $1<p<+\infty $, $s>\mu -1+1/p$.
\end{proof}

In the above proof, our standing assumption of ${\mathbb{D}}$-ellipticity
of ${\mathcal{A}}$ with respect to $\Lambda $ was not needed. A result similar
to \reftext{Proposition~\ref{prop:independent}} holds for model cone operators.

\begin{proposition}%
\label{prop:SI_for_model_cone}
Suppose that $0\neq\lambda \in \Lambda $ and that, in addition to the
${\mathbb{D}}$-ellipticity of ${\mathcal{A}}$, the conormal symbol of
${\mathcal{A}}$ is invertible on the line
$\revtex z = \frac{n+1}{2}-\gamma $. If
\begin{equation}
\label{prop:si_modelcone}
\lambda -\underline{\widehat{A}}_{\widehat{T}} : {\mathcal{K}}^{s+\mu ,
\gamma +\mu }_{p}(Y^{\wedge }, E^{\wedge })_{\widehat{T}} \oplus
\underline{\widehat{{\mathscr{E}}}} \longrightarrow {\mathcal{K}}^{s,
\gamma }_{p}(Y^{\wedge }, E^{\wedge })
\end{equation}
is invertible for some choice of $s$ and $p$, $s\in {\mathbb{Z}}$,
$s>\mu -1+1/p$, $1<p<\infty $, then it is invertible for all other choices.
\end{proposition}

\begin{proof}
The ${\mathbb{D}}$-ellipticity together with the fact that
$\lambda \neq0$ implies that
\begin{align*}
\binom{\lambda -\widehat{A}}{\widehat{T}}:{\mathcal{K}}^{s+\mu ,\gamma +
\mu }_{p}(Y^{\wedge }, E^{\wedge }) \longrightarrow
\begin{matrix}
{\mathcal{K}}^{s,\gamma }_{p}(Y^{\wedge }, E^{\wedge })
\\
\oplus
\\
\mathop{\text{$\oplus $}}\limits _{j=0}^{\mu -1} {\mathcal{K}}_{pp}^{s+
\mu -j-1/p,\mu -j-1/2}((\partial Y)^{\wedge },F_{j}^{\wedge })
\end{matrix}
\end{align*}
is a Fredholm operator for all above choices of $p$ and $s$: This follows
from Theorem 6.2.19 in \cite{HS}, since, after normalization of the orders
of the boundary operators, $\binom{\lambda -\widehat{A}}{\widehat{T}}$ is
an elliptic element in the cone algebra on the infinite cone.

Since $\widehat{T}$ is surjective, we obtain the Fredholm property for
\begin{equation*}
\lambda - \widehat{A}_{\widehat{T}}: {\mathcal{K}}^{s+\mu ,\gamma +\mu }_{p}(Y^{\wedge }, E^{\wedge })_{\widehat{T}} \longrightarrow {\mathcal{K}}^{s,
\gamma }_{p}(Y^{\wedge }, E^{\wedge })
\end{equation*}
and thus for the extension in \reftext{\eqref{prop:si_modelcone}}. The kernel of
this extension is actually independent of $s$ and $p$: Suppose
$(\lambda -\underline{\widehat{A}}_{\widehat{T}})(u+e)=0$ for some
$u\in {\mathcal{K}}^{s+\mu ,\gamma +\mu }_{p}(Y^{\wedge }, E^{\wedge })_{
\widehat{T}}$ and $e\in \widehat{{\mathscr{E}}}$. Then
$(\lambda - \widehat{A})u = (\lambda -\widehat{A}) e \in {\mathscr{S}}^{
\gamma +\varepsilon }(Y^{\wedge }, E^{\wedge })$ for sufficiently small
$\varepsilon >0$, so that, by elliptic regularity,
$u \in {\mathscr{S}}^{\gamma +\mu +\varepsilon '}(Y^{\wedge }, E^{\wedge })$
for some $\varepsilon >0$, which is a common subset of all domains, independent
of $s$ and $p$.

Furthermore, the invertibility of the conormal symbol on
$\revtex z = \frac{n+1}{2}-\gamma $ implies that the formal adjoint
$\widehat{A}^{t}$ of $\widehat{A}$ with the adjoint boundary condition
$\widehat{T}'$ is also ${\mathbb{D}}$-elliptic according to
\cite[Corollary 7.3]{CSS1}. An analog of \cite[Theorem 4.6]{CSS1} shows
that the adjoint of $\lambda -\underline{\widehat{A}}_{\widehat{T}}$ is
$\overline{\lambda }-\underline{\widehat{A}}^{t}_{\widehat{T}'}$ acting
on a domain of the form
${\mathcal{K}}^{s+\mu ,\gamma +\mu }_{p}(Y^{\wedge }, E^{\wedge })_{
\widehat{T}'}\oplus \underline{\widehat{{\mathscr{E}}}}^{t}$ for a suitable
$\underline{\widehat{{\mathscr{E}}}}^{t}$. With the same argument as above,
its kernel also is independent of $s$ and $p$.

Hence the index of $\lambda -\underline{\widehat{A}}_{\widehat{T}}$ is always
zero, as this is the case for the choice of $s$ and $p$, where it is invertible.
Moreover, since the kernel dimension is also constant, it must be zero.
Therefore $\lambda -\underline{\widehat{A}}_{\widehat{T}}$ is invertible
for all $s$ and $p$.
\end{proof}

\section{Parameter-dependent Green operators}
\label{sect:4}

Green symbols are parameter-dependent families of integral operators on
the model cone with smooth kernels that depend in a specific way on the
covariable $\eta $. We will show that they can be characterized by their
mapping properties, a result that will be needed in the proof of \reftext{Theorem~\ref{Resolvent}}. In the sequel, $[\cdot ]$ denotes a smooth function on
${\mathbb{R}}^{n}$ with $[\eta ] \equiv 1$ near zero and
$[\eta ]= |\eta |$ for $|\eta |\ge 1$.

\begin{definition}%
\label{def:RnuG}
Let $\nu ,\gamma _{0},\gamma _{1}\in \mathbb{R}$. Then
$\mathcal{R}_{G}^{\nu ,0}(Y^{\wedge },\Sigma ;\gamma _{0},\gamma _{1})$
consists of all operator families $a(\eta )$ of the form
\begin{equation*}
(a(\eta )u)(x,y)=[\eta ]^{n+1}\int _{0}^{\infty }\int _{Y}k_{a}(\eta ,x[
\eta ],y,s[\eta ],t)u(s,t)s^{n}dsdt
\end{equation*}
with integral kernel satisfying, for some $\varepsilon >0$,
\begin{equation}
\label{eq:kernel}
k_{a}(\eta ,x,y,s,t)\in S^{\nu }(\Sigma ,\mathscr{S}^{\gamma _{1}+
\varepsilon }(Y^{\wedge },E^{\wedge })\widehat{\otimes }_{\pi }
\mathscr{S}^{-\gamma _{0}+\varepsilon }(Y^{\wedge },E^{\wedge })).
\end{equation}
\end{definition}

For better readability, we do not mention the vector bundle
$E^{\wedge }$ in the notation. In \reftext{\eqref{eq:kernel}},
$\widehat{\otimes }_{\pi }$ denotes the completed projective tensor product
of Fr\'{e}chet spaces. 

For a Fr\'{e}chet space $F$,
$S^{\nu }(\Sigma ,F)$ denotes the space of $F$-valued symbols of order
$\nu $ on $\Sigma $, i.e. the smooth functions $a:\Sigma \to F$ such that,
for every multi-index $\alpha $ and every continuous semi-norm $q$ on
$F$, there exists a constant $C_{\alpha ,q}$ with
\begin{equation*}
q(D^{\alpha }_{\eta }a(\eta )) \le C_{\alpha ,q} [\eta ]^{\nu -|\alpha |},
\qquad \eta \in \Sigma .
\end{equation*}

While $\nu $ in ${\mathcal{R}}_{G}^{\nu ,0}$ has the interpretation of the
order of symbols, the parameter $0$ refers to the class or type of singular
Green operators in Boutet de Monvel's algebra.

Green symbols behave naturally under composition: If
$a_{j}\in {\mathcal{R}}_{G}^{\nu _{j},0}(Y^{\wedge },\Sigma ;\gamma _{j},
\gamma _{j+1})$, $j=0,1$, then
$a_{1}a_{0}\in {\mathcal{R}}_{G}^{\nu _{0}+\nu _{1},0}(Y^{\wedge },
\Sigma ;\gamma _{0},\gamma _{2})$.

If $\phi \in {\mathscr{C}}^{\infty }_{c}(Y^{\wedge })$ and
$a\in \mathcal{R}_{G}^{\nu ,0}(Y^{\wedge },\Sigma ;\gamma _{0},\gamma _{1})$,
then both $a\phi $ and $\phi a$ belong to
$\mathcal{R}_{G}^{-\infty ,0}(Y^{\wedge },\Sigma ;\break \gamma _{0},\gamma _{1})$,
i.e., are rapidly decreasing in the parameter.

For further details on Green symbols see also Schrohe, Schulze
\cite{SS_3}.

\begin{definition}%
\label{CGD}
The space
${\mathcal{C}}^{\nu ,0}_{G}({\mathbb{D}},\Sigma ;\gamma _{0},\gamma _{1})$
of parameter-dependent Green operators of order $\nu $ and class zero on
${\mathbb{D}}$ consists of all operator families $g$ of the form
\begin{equation*}
g(\eta )=\omega _{1}\,a(\eta )\,\omega _{0}+r(\eta ),
\end{equation*}
where $\omega _{0},\omega _{1}$ are cut-off functions,
$a\in {\mathcal{R}}^{\nu }_{G}(\Sigma ;\gamma _{0},\gamma _{1})$, and
$r$ has an integral kernel which, for some
$\varepsilon =\varepsilon (g)>0$, belongs to
\begin{equation}
\label{eq:ker_r}
{\mathscr{S}}(\Sigma ,{\mathscr{C}}^{\infty ,\gamma _{1}+\varepsilon }({
\mathbb{D}},E)\widehat{\otimes }_{\pi }\, {\mathscr{C}}^{\infty ,-\gamma _{0}+
\varepsilon }({\mathbb{D}},E)).
\end{equation}
\end{definition}

In the representation of $g$ above, the cut-off functions can be changed
at the cost of substituting $r$ by another element of the same structure.
Composition behaves as above, i.e., if
$g_{j}\in {\mathcal{C}}^{\nu _{j},0}_{G}({\mathbb{D}},\Sigma ;\gamma _{j},
\gamma _{j+1})$, $j=0,1$, then
$g_{1}g_{0}\in {\mathcal{C}}^{\nu _{1}+\nu _{2},0}_{G}({\mathbb{D}},
\Sigma ;\gamma _{0},\gamma _{2})$.

For later reference let us also note the following:

\begin{remark}%
\label{rem:change_omega}
If
$g\in {\mathcal{C}}_{G}^{\nu ,0}(\mathbb{D},\Sigma ;\gamma _{0},
\gamma _{1})$ and $\omega $ is a cut-off function, then
$\omega g \omega \in \mathcal{R}_{G}^{\nu ,0}(Y^{\wedge },\Sigma ;
\gamma _{0},\break \gamma _{1})$. Moreover, $(1-\omega ) g$ and
$g(1-\omega )$ belong to
${\mathcal{C}}^{-\infty ,0}_{G}({\mathbb{D}},\Sigma ;\gamma _{0},
\gamma _{1})$ and therefore have an integral kernel \reftext{\eqref{eq:ker_r}}.
\end{remark}

We now come to the characterization of Green symbols in terms of mapping
properties and symbol estimates rather than through the structure of their
integral kernels. To this end we briefly recall the concept of operator-valued
pseudodifferential symbols in spaces with group action.

A group action on a Banach space $X$ is a strongly continuous map
$\kappa :{\mathbb{R}}_{+}\to {\mathscr{L}}(X)$ with
$\kappa _{\lambda }\kappa _{\rho }=\kappa _{\lambda \rho }$ for every
$\lambda ,\rho >0$ and $\kappa _{1}=1$. All function and distribution spaces
over $Y^{\wedge }$ appearing in this paper will have the same group action,
defined by
\begin{equation}
\label{eq:standard-kappa}
\kappa _{\lambda }u(x,y) = \lambda ^{(n+1)/2} \, u(\lambda x,y);
\end{equation}
the factor $\lambda ^{(n+1)/2}$ makes this a group of unitary operators
in ${\mathcal{K}}^{0,0}_{2}(Y^{\wedge },E^{\wedge })$, which is the
$L^{2}$ space with respect to the cone-degenerate metric
$dx^{2}+x^{2}h(0)$.

Given two Banach spaces $X^{0},X^{1}$ with respective group actions
$\kappa ^{0},\kappa ^{1}$, we denote by
$S^{\nu }(\Sigma ;X^{0},X^{1})$ the space of all smooth
$a:\Sigma \to {\mathscr{L}}(X^{0},X^{1})$ such that
\begin{equation*}
\|{\kappa ^{1}_{1/[\eta ] }}D_{\eta }^{\alpha }a(\eta ) \kappa ^{0}_{[
\eta ]} \|_{{\mathscr{L}}(X^{0},X^{1})}\le C_{\alpha } [\eta ]^{\nu -|
\alpha |},\quad \eta \in \Sigma ,
\end{equation*}
with some constants $C_{\alpha }$. If $X^{1}$ is the projective limit of
Banach spaces
$X^{1}_{0}\supset X^{1}_{1}\supset X^{1}_{2}\supset \ldots $ and the restriction
of the group action in $X^{1}_{0}$ yields a group action of the other spaces,
we set
\begin{equation*}
S^{\nu }(\Sigma ;X^{0},X^{1})=\mathop{\text{\Large $\cap $}}_{j\in {
\mathbb{N}}} S^{\nu }(\Sigma ;X^{0},X^{1}_{j}).
\end{equation*}
In both cases, $S^{\nu }(\Sigma ;X^{0},X^{1})$ is a Fr\'{e}chet space in
a natural way.

A smooth function
$a:\Sigma \setminus \{0\}\to {\mathscr{L}}(X^{0},X^{1})$ is called homogeneous
of degree $\nu $, if
\begin{equation*}
\kappa _{1/\rho }^{1} a(\rho \eta ) \kappa ^{0}_{\rho }= \rho ^{\nu }a(
\eta ), \quad \eta \neq0\quad \rho >0.
\end{equation*}
Similarly,
$a\in {\mathscr{C}}^{\infty }(\Sigma ,{\mathscr{L}}(X^{0},X^{1}))$ is called
homogeneous of degree $\nu $ for large $|\eta |$, if there exists an
$R>0$ such that the above relation holds for $|\eta |\ge R$ and
$\rho \ge 1$. In this case $a$ belongs to
$S^{\nu }(\Sigma ;X^{0},X^{1})$.

\begin{proposition}%
\label{prop:green}
The following two properties are equivalent\emph{:}
\begin{enumerate}
\item[(1)]
$a\in \mathcal{R}^{\nu ,0}_{G}(Y^{\wedge }, \Sigma ;\gamma _{0},\gamma _{1})$
\item[(2)]
\label{eq:aa*}%
There exists an $\varepsilon >0$ such that
$a\in S^{\nu }(\Sigma ;{\mathcal{K}}^{0,\gamma _{0}}_{2}(Y^{\wedge },E^{\wedge }),{\mathscr{S}}^{\gamma _{1}+\varepsilon }(Y^{\wedge },E^{\wedge }))$
and
$a^{*}\in S^{\nu }(\Sigma ;{\mathcal{K}}^{0,-{\gamma _{1}}}_{2}(Y^{\wedge },E^{\wedge }), {\mathscr{S}}^{-\gamma _{0}+\varepsilon }(Y^{\wedge },E^{\wedge }))$.
\end{enumerate}
In $(2)$, the pointwise adjoint refers to the pairings induced by the inner
product of ${\mathcal{K}}^{0,0}_{2}(Y^{\wedge }, E^{\wedge })$. The group action
is given by \reftext{\eqref{eq:standard-kappa}}.
\end{proposition}

The proof is analogous to that of \cite[Proposition 4.6]{SS2} for the case
without boundary.

This class differs from the class
$R^{\nu ,0}_{G}(Y^{\wedge },\Sigma ;\gamma _{0},\gamma _{1})$ introduced
in Section~\ref{app:cone}, where we require the symbols to be classical.
The same applies to
${\mathcal{C}}^{\nu ,0}_{G}({\mathbb{D}},\Sigma ;\gamma _{0},\gamma _{1})$
in \reftext{Definition~\ref{CGD}} and
$C^{\nu ,0}_{G}({\mathbb{D}},\Sigma ;\gamma _{0},\gamma _{1})$ in \reftext{Definition~\ref{CGD2}}.

\section{Bounded $H_{\infty }$-calculus for parameter-elliptic realizations}
\label{sec:hinfty}

\subsection{Assumptions}
\label{sec5.1}

In the sequel we fix $\gamma \in {\mathbb{R}}$, $1<p<+\infty $, and a function
$\tilde{x}:\mathbb{D}\rightarrow \mathbb{R}$, equal to $x$ on
$[0,1)\times Y$ and strictly positive otherwise. We assume that
${\mathcal{A}}=\binom{A}{T}$ satisfies the following conditions. As before
$\Lambda $ denotes the sector \reftext{\eqref{Lambda}}.
\begin{itemize}
\item[(E1)] ${\mathcal{A}}$ is ${\mathbb{D}}$-elliptic with parameter in
the sense of \reftext{Definition~\ref{def:parameter-elliptic}}.
\item[(E$2_{\gamma }$)] The conormal symbol
$\pmb{\sigma }_{M}^{\mu }\binom{A}{T}(z)$ is invertible whenever
$\revtex z=\frac{n+1}{2}-\gamma -\mu $ or
$\revtex z = \frac{n+1}{2}-\gamma $.
\end{itemize}
We now consider a closed extension $\underline{A}_{T}$ of $A$ with domain
\begin{equation*}
{\mathscr{D}}(\underline{A}_{T}) = {\mathcal{H}}^{\mu ,\gamma +\mu }_{p}({
\mathbb{D}},E)_{T} \oplus \underline{{\mathscr{E}}}.
\end{equation*}

In order to simplify the analysis, we conjugate by
$\tilde{x}^{\gamma }$ and work with the weight $\gamma =0$, i.e. we define
the operator $A_{0} = \tilde{x}^{-\gamma }A\tilde{x}^{\gamma }$ and the boundary
condition $T_{0} = \tilde{x}^{-\gamma } T \tilde{x}^{\gamma }$ and study the
closed extension $\underline{A}_{0;T_{0}}$ of $A_{0}$ with domain
\begin{equation*}
{\mathscr{D}}(\underline{A}_{0;T_{0}}) = {\mathcal{H}}^{\mu ,\mu }_{p}({
\mathbb{D}},E)_{T_{0}} \oplus \tilde{x}^{-\gamma }
\underline{{\mathscr{E}}}.
\end{equation*}
The above conditions (E1) and (E$2_{\gamma }$) then take the form
\begin{itemize}
\item[(E1)] ${\mathcal{A}}_{0}$ is ${\mathbb{D}}$-elliptic with parameter
in the sense of \reftext{Definition~\ref{def:parameter-elliptic}}.
\item[(E2)] The conormal symbol
$\pmb{\sigma }_{M}^{\mu }\binom{A_{0}}{T_{0}}(z)$ is invertible whenever
$\revtex z =\frac{n+1}{2}-\mu $ or $\revtex z = \frac{n+1}{2}$.
\end{itemize}
We shall also consider the corresponding extension
$\underline{\widehat{A}}_{0,\widehat{T}_{0}}$ of the model cone operator
for $p=2$, i.e.
\begin{equation*}
{\mathscr{D}}(\underline{\widehat{A}}_{0,\widehat{T}_{0}}) = {\mathcal{K}}^{
\mu ,\mu }_{2}(Y^{\wedge }, E^{\wedge })_{\widehat{T}_{0}} \oplus x^{-
\gamma }\underline{\widehat{{\mathscr{E}}}},
\end{equation*}
cf. \reftext{Definition~\ref{def:domain_hat_A}}. We require that
\begin{itemize}
\item[(E3)] There exist $C, R\ge 0$ such that
$(\lambda -\underline{\widehat{A}}_{0,\widehat{T}_{0}})$ is invertible
for $\lambda $ in $\Lambda $, $|\lambda |\ge R$, and
\begin{equation*}
\|\lambda (\lambda - \underline{\widehat{A}}_{0,\widehat{T}_{0}})^{-1}
\|_{\mathcal{L}( \mathcal{K}^{0,0}_{2}(Y^{\wedge },E^{\wedge }))} \le C.
\end{equation*}
\end{itemize}

We state the condition in (E3) with $p=2$, because we will work with Hilbert
space adjoints in the proof. Also the fact that we require the invertibility
of the conormal symbol for $\revtex z =\frac{n+1}{2}$ is a consequence
of this technique.

\subsection{Special boundary conditions}
\label{sec5.2}

When working with a concrete realization $A_{T}$ of a boundary value problem,
it can be inconvenient to make the transition to $A_{0,T_{0}}$. In case
multiplication with a cut-off function $\omega $ commutes with the boundary
operator $T$ as it is the case e.g. for Dirichlet and Neumann boundary
conditions, this can be avoided. In order show this, we shall use a variant
of the ${\mathcal{K}}^{s,\gamma }_{p}$-spaces with weights at infinity: For
$\rho \in {\mathbb{R}}$ let
\begin{equation*}
{\mathcal{K}}^{s,\gamma }_{p}(Y^{\wedge }, E^{\wedge })^{\rho }= [x]^{-\rho } {
\mathcal{K}}^{s,\gamma }_{p}(Y^{\wedge },E^{\wedge }).
\end{equation*}

\begin{proposition}%
\label{prop:weight}
Suppose the boundary condition $\widehat{T}$ commutes with cut-off functions
$\omega $. Then
\begin{equation}
\label{eq:weight.1}
\lambda - {\widehat{A}}_{\widehat{T}}: {\mathcal{K}}^{\mu ,\gamma +\mu }_{2}(Y^{\wedge }, E^{\wedge })^{\rho }_{\widehat{T}}\oplus
\underline{\widehat{{\mathscr{E}}}} \longrightarrow {\mathcal{K}}_{2}^{0,
\gamma } (Y^{\wedge }, E^{\wedge })^{\rho
}\end{equation}
is invertible for any weight $\rho $ and $\lambda \in \Lambda $,
$|\lambda |$ sufficiently large, if and only it is invertible for
$\rho =0$ and $\lambda \in \Lambda $ sufficiently large. Moreover, if the
operator norm is $O(\lambda ^{-1})$ for some fixed $\rho $, then this is
the case for all $\rho $.
\end{proposition}

\begin{proof}
Suppose the operator $\lambda - {\widehat{A}}_{\widehat{T}}$ in \reftext{\eqref{eq:weight.1}} is invertible. By assumption, the boundary condition
$T$ is normal, hence so is $\widehat{T}$. Since normality is invariant under
conjugation by $[x]^{\rho }$ and implies surjectivity, we obtain the surjectivity
of
\begin{equation*}
\widehat{T}:{\mathcal{K}}^{\mu ,\gamma +\mu }_{2}(Y^{\wedge },E^{\wedge })^{\rho }\longrightarrow \mathop{\text{$\oplus $}}\limits _{j=0}^{\mu -1} {{
\mathcal{K}}}^{\mu -j-1/2,\gamma +\mu -j-1/2}_{22}(\partial Y^{\wedge },{F}_{j}^{\wedge })^{\rho }.
\end{equation*}

The invertibility of $\lambda -{\widehat{A}}_{\widehat{T}}$ in \reftext{\eqref{eq:weight.1}} therefore is equivalent to that of
\begin{align}
\label{eq:weight.2}
\binom{\lambda -\widehat{A}}{\widehat{T}}: {\mathcal{K}}^{\mu ,\gamma +
\mu }_{2}(Y^{\wedge },E^{\wedge })^{\rho }\oplus
\underline{\widehat{{\mathscr{E}}}} \longrightarrow
\begin{matrix}
{\mathcal{K}}^{0,\gamma }_{2}(Y^{\wedge })^{\rho
}\\
\oplus
\\
\mathop{\text{$\oplus $}}\limits _{j=0}^{\mu -1} {{\mathcal{K}}}^{\mu -j-1/2,
\gamma +\mu -j-1/2}_{22}(\partial Y^{\wedge },{F}_{j}^{\wedge })^{\rho },
\end{matrix}
\end{align}
see e.g. \cite[Corollary 8.2]{CSS1}. As a consequence,
\begin{equation*}
\binom{\lambda -\widehat{A}}{\widehat{T}}: {\mathcal{K}}^{\mu ,\gamma +
\mu }_{2}(Y^{\wedge },E^{\wedge })^{\rho }\longrightarrow
\begin{matrix}
{\mathcal{K}}^{0,\gamma }_{2}(Y^{\wedge })^{\rho
}\\
\oplus
\\
\mathop{\text{$\oplus $}}\limits _{j=0}^{\mu -1} {{\mathcal{K}}}^{\mu -j-1/2,
\gamma +\mu -j-1/2}_{22}(\partial Y^{\wedge },{F}_{j}^{\wedge })^{\rho
}\end{matrix}
\end{equation*}
is a Fredholm operator of index
$-N=-\dim \underline{\widehat{{\mathscr{E}}}} $, and the same is true for
\begin{align}
\label{eq:weight.3}
\binom{\lambda -\widehat{A}^{[\rho ]}}{\widehat{T}^{[\rho ]}}: {
\mathcal{K}}^{\mu ,\gamma +\mu }_{2}(Y^{\wedge },E^{\wedge })
\longrightarrow
\begin{matrix}
{\mathcal{K}}^{0,\gamma }_{2}(Y^{\wedge })
\\
\oplus
\\
\mathop{\text{$\oplus $}}\limits _{j=0}^{\mu -1} {{\mathcal{K}}}^{\mu -j-1/2,
\gamma +\mu -j-1/2}_{22}(\partial Y^{\wedge },{F}_{j}^{\wedge }),
\end{matrix}
\end{align}
where $\widehat{A}^{[\rho ]} = [x]^{\rho }\widehat{A} [x]^{-\rho }$ and
$\widehat{T}^{[\rho ]} = [x]^{\rho }\widehat{T} [x]^{-\rho }$. Also the operator
$\binom{\widehat{A}}{\widehat{T}}$ acts between the spaces in \reftext{\eqref{eq:weight.3}}. Moreover, the difference
$\binom{\widehat{A}}{\widehat{T}}-
\binom{\widehat{A}^{[\rho ]}}{\widehat{T}^{[\rho ]}}$ vanishes for small
$x$ and is an operator of order $\mu -1$. Hence, the difference is compact
between the spaces in \reftext{\eqref{eq:weight.3}}, and the operator
$\binom{\widehat{A}}{\widehat{T}}$ acting as in \reftext{\eqref{eq:weight.3}} also
has index $-N$. As a consequence,
\begin{equation*}
\binom{\lambda -\widehat{A}}{\widehat{T}}: {\mathcal{K}}^{\mu ,\gamma +
\mu }_{2}(Y^{\wedge },E^{\wedge })\oplus
\underline{\widehat{{\mathscr{E}}}}\longrightarrow
\begin{matrix}
{\mathcal{K}}^{0,\gamma }_{2}(Y^{\wedge })
\\
\oplus
\\
\mathop{\text{$\oplus $}}\limits _{j=0}^{\mu -1} {{\mathcal{K}}}^{\mu -j-1/2,
\gamma +\mu -j-1/2}_{22}(\partial Y^{\wedge },{F}_{j}^{\wedge }),
\end{matrix}
\end{equation*}
has index zero. In fact, it is invertible: In view of the ellipticity of
$\binom{\widehat{A}}{\widehat{T}}$, any function in its kernel is rapidly
decreasing as $x\to \infty $ and therefore also belongs to the kernel of
$\binom{\widehat{A}}{\widehat{T}}$ acting as in \reftext{\eqref{eq:weight.2}}, which
by assumption is $\{0\}$. The surjectivity of the boundary operator then
implies the invertibility of
\begin{equation*}
\lambda - \widehat{A}_{\widehat{T}}: {\mathcal{K}}^{\mu ,\gamma +\mu }_{2}(Y^{\wedge },E^{\wedge })_{\widehat{T}} \oplus
\underline{\widehat{{\mathscr{E}}}} \longrightarrow {\mathcal{K}}^{0,
\gamma }_{2}(Y^{\wedge },E^{\wedge }).
\end{equation*}

Let us finally check the estimates. If
$\|(\lambda -\widehat{A}_{\widehat{T}})^{-1}\|_{{\mathscr{L}}({
\mathcal{K}}^{0,\gamma }_{2}(Y^{\wedge }, E^{\wedge })^{\rho })} = O(
\lambda ^{-1})$, then also
\begin{equation}
\label{eq:weight.4}
\|[x]^{\rho } (\lambda -\widehat{A}_{\widehat{T}})^{-1} [x]^{-\rho } \|_{{
\mathscr{L}}({\mathcal{K}}^{0,\gamma }_{2} (Y^{\wedge },E^{\wedge }))} = O(
\langle \lambda \rangle ^{-1}).
\end{equation}
We can therefore consider the difference of
$(\lambda -\widehat{A}^{[\rho ]}_{\widehat{T}^{[\rho ]}})^{-1} = [x]^{
\rho } (\lambda -\widehat{A}_{\widehat{T}})^{-1} [x]^{-\rho }$ and
$(\lambda -\widehat{A}_{\widehat{T}})^{-1} $ as bounded operators in
${\mathcal{K}}^{0,\gamma }_{2}(Y^{\wedge },E^{\wedge })$. The range of
$[x]^{\rho } (\lambda -\widehat{A}_{\widehat{T}})^{-1} [x]^{-\rho }$ is the
image of the domain in \reftext{\eqref{eq:weight.1}} under multiplication by
$[x]^{\rho }$, i.e.,
\begin{equation*}
{\mathcal{K}}^{\mu ,\gamma +\mu }_{2}(Y^{\wedge },E^{\wedge })_{\widehat{T}^{[
\rho ]}}\oplus [x]^{\rho } \underline{\widehat{{\mathscr{E}}}}.
\end{equation*}
In view of the fact that the boundary condition $\widehat{T}$ commutes with
cut-off functions $\omega $, this is a subset of the maximal domain of
$\widehat{A}$: The projection $\widehat{R}\widehat{T}$ in
\reftext{(\ref{eq:complement-hat})} is not needed, so that
$\underline{\widehat{\mathscr{E}}} =\omega
\underline{\widehat{\mathscr{F}}}$. Since we can take $\omega $ to have
support so close to $x=0$ that $[x]\equiv 1$ on its support,
$[x]^{\rho } \underline{\widehat{{\mathscr{E}}}} =
\underline{\widehat{{\mathscr{E}}}}$. Hence $\widehat{A}$ maps the range
of $(\lambda -\widehat{A}^{[\rho ]}_{\widehat{T}^{[\rho ]}})^{-1}$ to
${\mathcal{K}}^{0,\gamma }(Y^{\wedge }, E^{\wedge })$ and we can write
\begin{equation*}
(\lambda -\widehat{A}_{\widehat{T}})^{-1}- (\lambda -\widehat{A}^{[\rho ]}_{
\widehat{T}^{[\rho ]}})^{-1} =(\lambda -\widehat{A}_{\widehat{T}})^{-1}(
\widehat{A}^{[\rho ]} - \widehat{A})(\lambda -\widehat{A}^{[\rho ]}_{
\widehat{T}^{[\rho ]}})^{-1}.
\end{equation*}
Since $A$ and $A^{[\rho ]}$ coincide for small $x$ and differ by an operator
of order $\mu -1$,
\begin{equation*}
\widehat{A}^{[\rho ]} - \widehat{A}: {\mathcal{K}}^{\mu -1,\gamma }_{2}(Y^{\wedge }, E^{\wedge })\longrightarrow {\mathcal{K}}^{0,\gamma }_{2}(Y^{\wedge }, E^{\wedge })
\end{equation*}
is bounded. On the other hand, by interpolation for the
${\mathcal{K}}^{*,*}_{2}$-spaces, see e.g. \cite[Lemma 4.1]{SS2}, we conclude
from \reftext{\eqref{eq:weight.4}} that
\begin{equation*}
\|(\lambda -\widehat{A}^{[\rho ]}_{\widehat{T}^{[\rho ]}})^{-1}\| _{{
\mathscr{L}}({\mathcal{K}}^{0,\gamma }_{2}(Y^{\wedge },E^{\wedge }), {
\mathcal{K}}^{\mu -1,\gamma }_{2}(Y^{\wedge },E^{\wedge }))} = O(\langle
\lambda \rangle ^{-1/\mu }).
\end{equation*}
This allows us to conclude that, for sufficiently large $|\lambda |$,
$\lambda \in \Lambda $,
\begin{equation*}
(\lambda -\widehat{A}_{\widehat{T}})^{-1} = (\lambda -\widehat{A}^{[
\rho ]}_{\widehat{T}^{[\rho ]}})^{-1} \Big ( 1-( \widehat{A}^{[\rho ]} -
\widehat{A}) (\lambda -\widehat{A}^{[\rho ]}_{\widehat{T}^{[\rho ]}})^{-1}
\Big )^{-1}.
\end{equation*}
Since the second factor on the right-hand side is bounded for large
$\lambda \in \Lambda $, we obtain the resolvent estimate for
$(\lambda -\widehat{A}_{\widehat{T}})^{-1} $ on
${\mathcal{K}}^{0,\gamma }_{2}(Y^{\wedge }, E^{\wedge })$.

Similarly, we can derive the estimate for any other $\rho $ from that for
$\rho =0$.
\end{proof}

\begin{corollary}%
\label{AltE3}
Suppose that $\widehat{T}$ commutes with cut-off functions. Then condition
\textup{(E3)} is equivalent to the following: There exists a $C\ge 0$ such
that for all $\lambda \in \Lambda $, $|\lambda |$ sufficiently large, the
realization $\lambda -\underline{\widehat{A}}_{\widehat{T}}$ is invertible
in ${\mathcal{K}}_2^{0,\gamma }(Y^{\wedge }, E^{\wedge })$ and
\begin{equation}
\label{E3Alternative}
\|(\lambda -\underline{\widehat{A}}_{\widehat{T}})^{-1}\|_{{\mathscr{L}}({
\mathcal{K}}^{0,\gamma }_{2}(Y^{\wedge }, E^{\wedge }))} = O(\lambda ^{-1}).
\end{equation}
\end{corollary}

\begin{proof}
\textup{(E3)} is equivalent to the invertibility of
$\lambda -\widehat{A}_{\widehat{T}}: {\mathcal{K}}_{2}^{\mu ,\gamma +\mu }(Y^{\wedge }, E^{\wedge })^{-\gamma }_{\widehat{T}}\oplus
\underline{\widehat{\mathscr{E}}}\rightarrow {\mathcal{K}}_{2}^{0,
\gamma }(Y^{\wedge },\break  E^{\wedge })^{-\gamma }$ with an
$O(\lambda ^{-1})$ estimate for
$(\lambda -\widehat{A}_{\widehat{T}})^{-1}$ on
${\mathcal{K}}_{2}^{0,\gamma }(Y^{\wedge }, E^{\wedge })^{-\gamma }$. By \reftext{Proposition~\ref{prop:weight}}, this is equivalent to the invertibility of
$\lambda -\widehat{A}_{\widehat{T}}: {\mathcal{K}}_{2}^{\mu ,\gamma +\mu }(Y^{\wedge }, E^{\wedge })_{\widehat{T}}\oplus
\underline{\widehat{\mathscr{E}}}\rightarrow {\mathcal{K}}_{2}^{0,
\gamma }(Y^{\wedge }, E^{\wedge })$ with an $O(\lambda ^{-1})$ estimate for
the inverse on ${\mathcal{K}}^{0,\gamma }_{2}(Y^{\wedge }, E^{\wedge })$.
\end{proof}

\subsection{Dilation invariant domains}
\label{subsec:dilation}

We call a subspace ${\mathscr{D}}$ of
${\mathcal{K}}^{0,\gamma }_{p}(Y^{\wedge }, E^{\wedge })$ dilation-invariant,
if $u\in {\mathscr{D}}$ implies that
$\kappa _{\rho }u \in {\mathscr{D}}$ for arbitrary $\rho >0$ with
$\kappa $ as in \reftext{\eqref{eq:standard-kappa}}.

Suppose ${\mathscr{D}}(\underline{\widehat{A}}_{\widehat{T}}) $ is dilation
invariant. Then, for all $\eta \in \Sigma $ and $\rho >0$,
\begin{equation*}
\eta ^{\mu }- \underline{\widehat{A}}_{\widehat{T}} = \rho ^{\mu }\kappa _{
\rho } ((\eta /\rho )^{\mu }-\underline{\widehat{A}}_{\widehat{T}})
\kappa _{\rho }^{-1}: {\mathscr{D}}(\underline{\widehat{A}}_{\widehat{T}})
\longrightarrow {\mathcal{K}}^{0,\gamma }_{p}(Y^{\wedge }, E^{\wedge }).
\end{equation*}

Choosing $\rho = |\eta |$, we conclude that
$\eta ^{\mu }-\underline{\widehat{A}}_{\widehat{T}}$ is invertible for large
$|\eta |$, $\eta \in \Sigma $, if and only if it is invertible for all
$\eta \in \Sigma \setminus \{0\}$, if and only if it is invertible for
all $\eta \in \Sigma $ with $|\eta |=1$.

The dilation invariance of
${\mathscr{D}}(\underline{\widehat{A}}_{\widehat{T}})$ in
${\mathcal{K}}^{0,\gamma }_{2}(Y^{\wedge }, E^{\wedge })$ implies that of
${\mathscr{D}}(\underline{\widehat{A}}_{0,\widehat{T}_{0}})$ in
${\mathcal{K}}^{0,0}_{2}(Y^{\wedge }, E^{\wedge })$. The fact that
$\kappa _{\rho }$ is unitary on
${\mathcal{K}}^{0,0}_{2}(Y^{\wedge }, E^{\wedge })$ then shows that
\begin{align*}
\|(\eta ^{\mu }-&\underline{\widehat{A}}_{0,\widehat{T}_{0}})^{-1}\|_{{
\mathscr{L}}({\mathcal{K}}^{0,0}_{2}(Y^{\wedge },E^{\wedge }))} = \|\kappa ^{-1}_{|
\eta |} (\eta ^{\mu }-\underline{\widehat{A}}_{0,\widehat{T}_{0}})^{-1}
\kappa _{|\eta |}\|%
_{{\mathscr{L}}({\mathcal{K}}^{0,0}_{2}(Y^{\wedge },E^{\wedge }))}
\\
&=|\eta |^{-\mu } \|((\eta /|\eta |)^{\mu }-\underline{\widehat{A}}_{0,
\widehat{T}_{0}})^{-1}\|%
_{{\mathscr{L}}({\mathcal{K}}^{0,0}_{2}(Y^{\wedge },E^{\wedge }))} \le C|
\eta |^{-\mu }
\end{align*}
provided the inverse on the right-hand side exists for
$\eta \in \Sigma $, $|\eta |=1$. We see:

\begin{proposition}%
For dilation invariant domains, condition \textup{(E3)} is equivalent to the
existence of $(\lambda -\underline{\widehat{A}}_{\widehat{T}} )^{-1}$ for
$\lambda \in \Lambda $, $|\lambda |=1$.
\end{proposition}

\subsection{The $H_{\infty }$-calculus}
\label{sec5.4}

In order to obtain a precise structure of the resolvent, we shall make
use of the pseudodifferential calculus for boundary value problems on manifolds
with edges as presented in Section 4 of \cite{KS} and Section 7 of
\cite{HS}. The basic elements are recalled in Section~\ref{app:cone} of
the appendix.

\begin{proposition}[Parametrix]%
\label{rp}
Assume that \textup{(E1), (E$2_{\gamma }$)} are fulfilled and that
\begin{align}
\label{modelconeAT}
\begin{pmatrix}
\eta ^{\mu }-\widehat{A}
\\
\widehat{T}
\end{pmatrix} : \mathcal{K}_{2}^{\mu ,\gamma +\mu }(Y^{\wedge },E^{\wedge })
\longrightarrow
\begin{matrix}
\mathcal{K}^{0,\gamma }_{2} (Y^{\wedge },E^{\wedge })
\\
\oplus
\\
\mathop{\oplus }\limits _{j=0}^{\mu -1} {\mathcal{K}}_{2}^{
\mu -j-1/2,\gamma +\mu -j-1/2}(\partial Y^{\wedge }, F_{k}^{\wedge })
\end{matrix}
\end{align}
is injective for $0\neq\eta \in \Sigma $. Then there exists a
$B \in C^{-\mu ,0}({\mathbb{D}},\Sigma ;\gamma ,\gamma +\mu )$ with the
following properties:
\begin{itemize}
\item[(i)]
$B(\eta )\in \mathcal{L}(\mathcal{H}_{2}^{0,\gamma }(\mathbb{D}),
\mathcal{H}_{2}^{\mu ,\gamma +\mu }(\mathbb{D})_{T})$ for every
$\eta \in \Sigma $, $|\eta |$ large.
\item[(ii)] $(\eta ^{\mu }-{A})B(\eta ) -1=G_{R}(\eta )$ with
$G_{R}\in C_{G}^{0,0}({\mathbb{D}},\Sigma ;\gamma ,\gamma )$.
\item[(iii)] $B(\eta )(\eta ^{\mu }-{A})=1+G_{L}(\eta )$ for some
$G_{L}\in C_{G}^{0,\mu }(\mathbb{D},\Sigma ;\gamma +\mu ,\gamma +\mu )$,
and $B(\eta )(\eta ^{\mu }-{A})=1$ on
${\mathcal{H}}^{\mu ,\gamma +\mu }_{p}({\mathbb{D}})_{T}$.
\end{itemize}
\end{proposition}

\begin{proof}
In order to unify the orders in the operator matrix, we replace $T$ by
$T_{1}(\eta )= \break \operatorname{diag} (R^{\mu -1/2}(\eta ), R^{\mu -3/2}(
\eta ),\ldots ,R^{1/2}(\eta ))T$, where $R^{\mu -j+1/2}(\eta )$ is a parameter-de\-pen\-dent
order reduction on ${\mathbb{B}}$, so that
\begin{equation*}
R^{\mu -j-1/2}(\eta ): \mathcal{K}_{2}^{\mu -j-1/2,\gamma +\mu -j-1/2}({
\mathbb{B}},F_{j}) \longrightarrow \mathcal{K}_{2}^{0,\gamma }({
\mathbb{B}},F_{j})
\end{equation*}
is an isomorphism. Clearly, $\ker T_{1}(\eta ) = \ker T$. Also
$\ker \widehat{T}_{1}(\eta )=\ker \widehat{T}$, since the edge symbols of
the order reducing operators are invertible.

Now $
\begin{pmatrix}
\eta ^{\mu }-A
\\
T_{1}(\eta )%
\end{pmatrix}
$ is a symbol of order $\mu $ and type $\mu $ in Schulze's parameter-dependent
cone calculus, and its principal edge symbol,
\begin{equation}
\begin{pmatrix}
\eta ^{\mu }-\widehat{A}
\\
\widehat{T}_{1}(\eta )
\end{pmatrix}
:\mathcal{K}_{2}^{\mu ,\gamma +\mu }(Y^{\wedge }, E^{\wedge })
\longrightarrow
\begin{matrix}
\mathcal{K}_{2}^{0,\gamma } (Y^{\wedge },E^{\wedge })
\\
\oplus
\\
{\mathcal{K}}_{2}^{0,\gamma }(\partial Y^{\wedge }, F^{\wedge })
\end{matrix}
\end{equation}
is injective. Following an argument by Krainer, see the proof of Theorem
7.21 in \cite{Kr}, we find an operator family
$\widehat{K}_{1}(\eta ): \mathbb{C}^{d}\to {\mathscr{S}}^{\infty }(Y^{\wedge }, E^{\wedge })$, $\eta \in \Sigma \, \cap \, \{|\eta |=1\}$, such
that
\begin{equation*}
\begin{pmatrix}
\eta ^{\mu }-\widehat{A}& \widehat{K}_{1}(\eta )
\\
\widehat{T}_{1}(\eta )&0
\end{pmatrix}
:
\begin{matrix}
\mathcal{K}_{2}^{\mu ,\gamma +\mu }(Y^{\wedge }, E^{\wedge })
\\
\oplus
\\
{\mathbb{C}}^{d}
\end{matrix}
\longrightarrow
\begin{matrix}
\mathcal{K}_{2}^{0,\gamma } (Y^{\wedge }, E^{\wedge })
\\
\oplus
\\
{\mathcal{K}}_{2}^{0,\gamma }(\partial Y^{\wedge }, F^{\wedge })
\end{matrix}
\end{equation*}
is invertible. In fact, $\widehat{K}_{1}$ can be chosen to have an integral
kernel in
\begin{equation*}
{\mathscr{C}}^{\infty }(\Sigma \,\cap \,\{|\eta |=1\})\widehat{\otimes }_{\pi }{\mathscr{C}}_{c}^{\infty }(Y^{\wedge },E^{\wedge }) \otimes \mathbb{C}^{d};
\end{equation*}
it can be extended $\mu $-homogeneously to $\Sigma $ by
\begin{equation*}
\widehat{K}_{1}(\eta ) = |\eta |^{\mu }\kappa _{|\eta |} \widehat{K}_{1}(
\eta /|\eta |) .
\end{equation*}
Let $K_{1}(\eta ) = \omega \chi (\eta ) \widehat{K}_{1}(\eta ) $ for a zero-excision
function $\chi $ and a cut-off function $\omega $. Then
\begin{align}
\label{Ainvertible}
\begin{pmatrix}
\eta ^{\mu }- A& K_{1}(\eta )
\\
T_{1}(\eta )&0
\end{pmatrix}:
\begin{matrix}
\mathcal{H}^{\mu ,\gamma +\mu }_{2}(\mathbb{D},E)
\\
\oplus
\\
{\mathbb{C}}^{d}
\end{matrix} \longrightarrow
\begin{matrix}
\mathcal{H}_{2}^{0,\gamma } (\mathbb{D},E)
\\
\oplus
\\
{\mathcal{H}}_{2}^{0,\gamma }({\mathbb{B}}, F)
\end{matrix}
\end{align}
is an elliptic element in Schulze's parameter-dependent cone calculus.
Hence there exists a parametrix modulo regularizing Green operators.

The operator in \reftext{\eqref{Ainvertible}} is invertible for large
$|\eta |$, and we can modify the parametrix so that it coincides with the
inverse for large $|\eta |$. Denote this parametrix by
$
\begin{pmatrix}
B(\eta ) & K(\eta )
\\
S(\eta ) & Q(\eta )%
\end{pmatrix}
$. Then
\begin{equation*}
\begin{pmatrix}
1&0
\\
0&1%
\end{pmatrix}
=
\begin{pmatrix}
\eta ^{\mu }- A& K_{1}(\eta )
\\
T_{1}(\eta )&0
\end{pmatrix}
\begin{pmatrix}
B(\eta ) & K(\eta )
\\
S(\eta ) & Q(\eta )%
\end{pmatrix}
\end{equation*}
shows that $B(\eta )$ maps into the kernel of $T$ and that
$(\eta ^{\mu }- A)B(\eta )-1=-K_{1}(\eta )S(\eta )$. This shows (i) and (ii).
Interchanging the order of factors, one obtains (iii).
\end{proof}

\begin{theorem}%
\label{Resolvent}
Let the conditions \textup{(E1)-(E3)} be satisfied for $A_{0}$ with boundary
condition $T_{0}$. Then $\underline{A}_{T}$ has at most finitely many spectral
points in $\Sigma $, and there exists a parameter-dependent operator
\begin{align}
\label{eq:inverse}
C\in C^{-\mu ,0}({\mathbb{D}},\Sigma ; \gamma , \gamma +\mu ) + {
\mathcal{C}}_{G}^{-\mu ,0}({\mathbb{D}},\Sigma ;\gamma ,\gamma )
\end{align}
such that $(\eta ^{\mu }-\underline{A}_{T})^{-1}=C(\eta )$ for sufficiently
large $\eta \in \Sigma $.
\end{theorem}

In the above decomposition \reftext{\eqref{eq:inverse}},
$C^{-\mu ,0}({\mathbb{D}},\Sigma ; \gamma , \gamma +\mu ) $ is the symbol
class introduced in Section~\ref{app:cone}, while
${\mathcal{C}}_{G}^{-\mu ,0}({\mathbb{D}},\Sigma ;\gamma ,\gamma )$ is as
in \reftext{Definition~\ref{CGD}}.

\begin{proof}
By conjugation with $\tilde{x}^{\gamma }$ it is sufficient to show the assertion
for $A_{0,T_{0}}$ or, equivalently, to assume that the weight
$\gamma $ is zero, so that $A_{0}=A$ and $T_{0}=T$.

It follows from \cite[Theorem 8.1]{Kr} that
$(\eta ^{\mu }-\underline{A}_{T})^{-1}$ exists for $\eta \in \Sigma $,
$|\eta |$ sufficiently large, and
\begin{equation*}
\|(\eta ^{\mu }-\underline{A}_{T})^{-1}\|_{{\mathscr{L}}({\mathcal{H}}_{2}^{0,0}({
\mathbb{D}}, E))}= O(\langle \eta \rangle ^{-\mu }).
\end{equation*}
Assumption (E3) implies that $\eta ^{\mu }-\widehat{A}_{\widehat{T}}$ is injective
on
${\mathscr{D}}({\widehat{A}_{\widehat{T} ,\min }}) =\mathcal{K}^{\mu ,\mu }_{2}(Y^{\wedge },E^{\wedge })_{\widehat{T}}$. Thus also the operator \reftext{\eqref{modelconeAT}} is injective. By \reftext{Proposition~\ref{rp}} we find
$B(\eta )\in C^{-\mu ,0}(\mathbb{D},\Sigma ;0,\mu )$ and
$G_{R}\in C_{G}^{0,0}({\mathbb{D}},\Sigma ;0,0)$ such that, for sufficiently
large $|\eta |$,
\begin{gather}
\label{fst}
B(\eta )=(\eta ^{\mu }-\underline{A}_{T})^{-1}+(\eta ^{\mu }-
\underline{A}_{T})^{-1}G_{R}(\eta ).
\end{gather}

Let $\widetilde{T}$ be the adjoint boundary condition for $A_{T}$ in the
sense of \cite[Definition 3.12]{CSS1}. Then
$(\underline{A}_{T})^{*} = \underline{A}^{t}_{\widetilde{T}}$, i.e. the adjoint
of $\underline{A}_{T}$ is a suitable realization of the formal adjoint
$A^{t}$ with boundary condition $\widetilde{T}$, see
\cite[Section 4]{CSS1}. Let us check that it also satisfies (E1) and (E2).

Clearly, parameter-ellipticity of the principal pseudodifferential symbol
holds for $A^{t}$. If we write the boundary symbol of $\binom{A}{T}$ as
$\binom{a^{0}}{t^{0}}$ and that of $\binom{A^{t}}{\widetilde{T}}$ as
$\binom{a^{t,0}}{{\tilde{t}}^{0}}$, then the boundary symbol realizations
in $L^{2}({\mathbb{R}}_{+})$ satisfy
$(a^{t,0})_{\tilde{t}^{0}}= ((a^{0})_{t^{0}})^{*}$, with a corresponding
relation for the rescaled symbols, see Grubb \cite[Theorem 1.6.9]{Gr}.
Therefore (E1) also holds for the adjoint.

Moreover, \cite[Corollary 7.3]{CSS1} applied with $\gamma =0$ and
$\gamma =\mu $ implies that the conormal symbol of
$\binom{A^{t}}{\widetilde{T}}(z)$ is invertible for
$\revtex z = \frac{n+1}{2}-\mu $ and $\revtex z = \frac{n+1}{2}$. Hence
(E2) holds for $\binom{A^{t}}{\widetilde{T}}$.

It follows from \cite[Proposition 7.2]{CSS1} that
$\widehat{\widetilde{T}}= \widetilde{\widehat{T}} $, since
$\widehat{T}$ and $\widehat{\widetilde{T}}$ are differential operators so
that the model cone boundary condition is determined by the conormal symbol.
Therefore $(\underline{\widehat{A}}_{\widehat{T}})^{*}$ is a realization
of ${\widehat{A^{t}}}$, subject to the boundary condition
$\widehat{\widetilde{T}}$. Since this realization has no spectrum in
$\Lambda \cap \{|\lambda |\ge R\}$, the operator
$\lambda - \widehat{A^{t}}_{\widehat{\widetilde{T}},\min }$ is injective
for large $\lambda \in \Lambda $, and thus, by homogeneity, for all
$\lambda \in \Lambda \setminus \{0\}$.

So we can apply once more \reftext{Proposition~\ref{rp}} and find
$\widetilde{B}(\eta )\in C^{-\mu ,0}(\mathbb{D},\Sigma ;0,\mu )$ and
$\widetilde{G}_{R}\in C_{G}^{0,0}(\mathbb{D},\Sigma ;0,0)$ such that
\begin{align*}
\widetilde{B}(\eta )=(\overline{\eta }^{\mu }-(\underline{A}_{T})^{*})^{-1}
+(\overline{\eta }^{\mu }-(\underline{A}_{T})^{*})^{-1}\widetilde{G}_{R}(
\eta ).
\end{align*}
Taking adjoints in the above equation we obtain
\begin{align*}
\widetilde{B}^{*}(\eta )=(\eta ^{\mu }-\underline{A}_{\widetilde{T}})^{-1}+
\widetilde{G}_{R}^{*}(\eta )(\eta ^{\mu }-\underline{A}_{T})^{-1}.
\end{align*}
Hence
\begin{equation*}
(\eta ^{\mu }-\underline{A}_{T})^{-1} = B(\eta ) + \widetilde{B}^{*}G_{R} -
\widetilde{G}_{R}^{*}(\eta )(\eta ^{\mu }-\underline{A}_{T})^{-1}G_{R}(
\eta ).
\end{equation*}
By the rules of the calculus,
$\widetilde{B}^{*}G_{R} \in C^{-\mu ,0}_{G}(\mathbb{D},\Sigma ;,0,0)$;
so it remains to show that
$ \widetilde{G}_{R}^{*}(\eta )(\eta -\underline{A}_{T})^{-1}G_{R}(\eta )
\in {\mathcal{C}}^{-\mu ,0}_{G}({\mathbb{D}},\Sigma ;0,0)$.

From the fact that
$D^{\alpha }_{\eta }(\eta ^{\mu }-\underline{A}_{T})^{-1}$ is a linear combination
of terms of the form
$p_{kl}(\eta ) (\eta ^{\mu }-\underline{A}_{T})^{-1-l}$, where
$k+l=|\alpha |$ and $p_{kl}$ is a polynomial of degree at most
$(\mu -1)l-k$, we conclude that, for
$\alpha \in {\mathbb{N}}_{0}^{2}$, there exists a constant
$C_{\alpha }$ with
\begin{equation*}
\|D^{\alpha }_{\eta }(\eta ^{\mu }-\underline{A}_{T})^{-1}\|_{{\mathscr{L}}({
\mathcal{H}}^{0,0}_{2}({\mathbb{D}},E))} \le C_{\alpha }[\eta ]^{-\mu -|
\alpha |}.
\end{equation*}
Since the group action $\kappa $ is unitary on
${\mathcal{K}}^{0,0}_{2}(Y^{\wedge },E^{\wedge })$, we see that, for any cut-off
function $\omega $,
\begin{equation*}
\omega (\eta ^{\mu }-\underline{A}_{T})^{-1}\omega \in S^{-\mu }(\Sigma ; {
\mathcal{K}}^{0,0}_{2}(Y^{\wedge },E^{\wedge }), {\mathcal{K}}^{0,0}_{2}(Y^{\wedge },E^{\wedge })).
\end{equation*}
Furthermore,
$\omega G_{R}(\eta ) \omega , \omega \widetilde{G}^{*}_{R}\omega \in
\mathcal{R}_{G}^{0,0}(Y^{\wedge },\Sigma ; 0,0)$, so that
\begin{equation*}
R(\eta ) :=\omega \widetilde{G}_{R}^{*}(\eta )\omega ^{2} (\eta ^{\mu }-
\underline{A}_{T})^{-1} \omega ^{2} G_{R}(\eta ) \omega \in S^{-\mu }(
\Sigma ; {\mathcal{K}}^{0,0}_{2} (Y^{\wedge }, E^{\wedge }), {\mathscr{S}}^{
\varepsilon }(Y^{\wedge }, E^{\wedge })).
\end{equation*}
By \reftext{Remark~\ref{rem:change_omega}} we may omit the two factors
$\omega ^{2}$ at the expense of modifying $R$ by an element of
$S^{-\infty }(\Sigma ; {\mathcal{K}}^{0,0}_{2} (Y^{\wedge }, E^{\wedge }), {
\mathscr{S}}^{\varepsilon }(Y^{\wedge }, E^{\wedge }))$. In particular, this
preserves the symbol class on the right-hand side. As
${\mathcal{R}}_{G}^{0,0}(Y^{\wedge }, \Sigma ; 0,0)$ is invariant under adjoints,
the argument applies also to the adjoints. So $R^{*}$ and its modification
belong to
$S^{-\mu }(\Sigma ; {\mathcal{K}}_{2}^{0,0} (Y^{\wedge }, E^{\wedge }), {
\mathscr{S}}^{\varepsilon }(Y^{\wedge }, E^{\wedge }))$, and \reftext{Proposition~\ref{prop:green}} implies that
\begin{equation*}
\omega \, \widetilde{G}_{R}^{*}\,(\eta ^{\mu }-\underline{A}_{T})^{-1} G_{R}
\,\omega \in {\mathcal{R}}_{G}^{-\mu ,0}(Y^{\wedge },\Sigma ; 0,0).
\end{equation*}
We conclude again from \reftext{Remark~\ref{rem:change_omega}} that
$ \widetilde{G}_{R}^{*}\,(\eta ^{\mu }-\underline{A}_{T})^{-1} G_{R} \in {
\mathcal{C}}_{G}^{-\mu ,0}({\mathbb{D}},\Sigma ;0,0)$, and the proof is complete.
\end{proof}

The following theorem is a consequence of \reftext{Theorem~\ref{Resolvent}} and
\cite[Theorem 4.1]{CSS0}, noting that the required holomorphicity of the
principal interior symbol is immediate from the fact that it arises as
the inverse of the principal symbol of $\eta ^{\mu }-A$. This has been shown
already in the proof of \cite[Theorem 5.4]{CSS0}.

\begin{theorem}%
\label{thm:main}
Under the assumptions \textup{(E1)-(E3)} there exists a constant $c>0$ such
that $c+\underline{A}_{T}$ has a bounded $H_{\infty }$-calculus on
${\mathcal{H}}^{0,\gamma }_{p}({\mathbb{D}},E)$ for all $1<p<+\infty $.
\end{theorem}

\section{The Dirichlet and Neumann Laplacian}
\label{sec:Laplacian}

Given a metric that coincides with $dx^{2}+x^{2}h(x)$ on the collar part
$(0,1)\times Y$ of ${\mathbb{D}}$, the associated Laplacian is, on the collar
part,
\begin{equation*}
\Delta =x^{-2}\Big ((x\partial _{x})^{2}+(n-1+H(x))(x\partial _{x})+
\Delta _{Y}(x)\Big ),
\end{equation*}
where $\Delta _{Y}(x)$ is the Laplacian on $Y$ induced by $h(x)$ and
$2H(x) = x\partial _{x}(\log \det h(x))$.

We shall consider realizations subject to the Dirichlet boundary operator
$\gamma _{D}=\gamma _{0}$ and to the Neumann boundary operator
$\gamma _{N}=\gamma _{0} B_{1}$, where
$B_{1}=x^{-1}\partial _{\nu _{x}}$ in the sense of \reftext{\eqref{eq:A}} and
$\nu _{x}$ is a unit vector field in a collar-neighborhood of
$\partial Y$ that coincides on $\partial Y$ with the exterior normal with
respect to the metric $h(x)$. According to Green's formula, both
$\Delta _{D} := \Delta _{\gamma _{D}}$ and
$\Delta _{N} :=\Delta _{\gamma _{N}}$ are symmetric in
${\mathcal{H}}^{0,0}_{2}({\mathbb{D}})$. They have been described in
\cite{CSS1} in the special case of a metric that is constant in $x$, i.e.,
$h(x)\equiv h(0)$.

We write ${\mathcal{A}}_{D}=\binom{\Delta }{\gamma _{D}}$ and
${\mathcal{A}}_{N}=\binom{\Delta }{\gamma _{N}}$ and obtain the model cone
operators
\begin{equation*}
\widehat{\Delta }=x^{-2}\Big ((x\partial _{x})^{2}+(n-1)(x\partial _{x})+
\Delta _{Y}(0)\Big ),
\end{equation*}
$\widehat{{\mathcal{A}}}_{D}=
\binom{\widehat{\Delta }}{\widehat{\gamma }_{D}}$ and
$\widehat{\mathcal{A}}_{N}=
\binom{\widehat{\Delta }}{\widehat{\gamma }_{N}}$, with the Dirichlet and
Neumann boundary operators $\widehat{\gamma }_{D}$ and
$\widehat{\gamma }_{N}$, respectively, on $Y^{\wedge }$ equipped with the
metric $dx^{2}+x^{2}h(0)$. Both commute with multiplication by cut-off
functions, cf. \reftext{Remark~\ref{rem:RT}}.

The resulting principal conormal symbols are
\begin{equation*}
\mathbf{f}_{0,D/N}(z)=\sigma _{M}^{2}({\mathcal{A}}_{D/N})(z) =
\begin{pmatrix}
z^{2}-(n-1)z+\Delta _{Y}(0)
\\
\gamma _{D/N}%
\end{pmatrix}
.
\end{equation*}
Denoting by $\Delta _{Y,D/N}(0)$ the Dirichlet respectively Neumann realization
of the Laplacian on $Y$, let us now set, for $z\in {\mathbb{C}}$,
\begin{equation*}
f_{0}(z)=z^{2}-(n-1)z+\Delta _{Y}(0),\qquad f_{0,D/N}(z)=z^{2}-(n-1)z+
\Delta _{Y,D/N}(0).
\end{equation*}

\begin{lemma}%
\label{lem:inv}
Let $K_{D/N}$ be a right-inverse of $\gamma _{D/N}$ in Boutet de Monvel's
algebra for $Y$. For $z\in {\mathbb{C}}$,
\begin{equation*}
\mathbf{f}_{0,D}(z): H^{s+2}_{2}(Y)\longrightarrow H^{s}_{2}(Y)
\oplus H^{s+3/2}_{2}(\partial Y), \quad s>-3/2,
\end{equation*}
and
\begin{equation*}
\mathbf{f}_{0,N}(z): H^{s+2}_{2}(Y)\longrightarrow H^{s}_{2}(Y)
\oplus H^{s+1/2}_{2}(\partial Y) ,\quad s> -1/2,
\end{equation*}
respectively, are invertible if and only if
\begin{equation*}
f_{0,D/N}(z): H^{s+2}_{2}(Y)_{D/N} \longrightarrow H^{s}_{2}(Y)
\end{equation*}
are invertible; in this case
\begin{align*}
\mathbf{f}_{0,D/N}(z)^{-1} & =
\begin{pmatrix}
f_{0,D/N}(z)^{-1} & (1-f_{0,D/N}(z)^{-1}f_{0}(z))K_{D/N}
\end{pmatrix}
\\
& = f_{0,D/N}(z)^{-1}
\begin{pmatrix}
1 & -f_{0}(z)K_{D/N}
\end{pmatrix} +
\begin{pmatrix}
0&K_{D/N}
\end{pmatrix}.
\end{align*}
\end{lemma}
\begin{proof}
The first fact follows from the surjectivity of the boundary operators,
see \cite[Corollary 8.2]{CSS1}. The formula for the inverse then results
from the identity
\begin{equation*}
\begin{pmatrix}
f_{0}(z)
\\
\gamma _{D/N}%
\end{pmatrix}
\begin{pmatrix}
f_{0,D/N}(z)^{-1} & K_{D/N}%
\end{pmatrix}
=
\begin{pmatrix}
1&f_{0}(z)K_{D/N}
\\
0 & 1%
\end{pmatrix}
.\qedhere
\end{equation*}
\end{proof}

\begin{corollary}%
\label{cor:G_D/N}
In \reftext{Theorem~\textup{\ref{thm:theta}}}, applied to ${\mathcal{A}}_{D}$ or
${\mathcal{A}}_{N}$, the operators $\mathbf{G}^{(\ell )}_{\sigma }$ can be
substituted by the operators
\begin{equation*}
G^{(\ell )}_{\sigma ,D/N}:{\mathscr{S}}^{\infty }(Y^{\wedge })
\longrightarrow {\mathscr{C}}^{\infty }(Y^{\wedge })
\end{equation*}
defined by
\begin{equation}
\label{eq:gsigmal-D/N}
(G_{\sigma ,D/N}^{(\ell )} v)(x)= x^{\ell }\,\int _{|z-\sigma |=
\varepsilon } x^{-z}\mathbf{g}_{\ell ,D/N}(z)\,\Pi _{\sigma }(f_{0,D/N}^{-1}
\,\widehat{v})(z)\,\dbar z.
\end{equation}
\end{corollary}
\begin{proof}
By \reftext{Lemma~\ref{lem:inv}},
$\mathbf{f}_{0,D/N}(z)^{-1}\equiv f_{0,D/N}(z)^{-1}
\begin{pmatrix}
1 & -f_{0}(z)K_{D/N}%
\end{pmatrix}
$ modulo holomorphic functions. Now let
$u=(u_{1},u_{2})\in {\mathscr{S}}^{\infty }(Y^{\wedge })\oplus {\mathscr{S}}^{\infty }(\partial Y^{\wedge })$. Then
\begin{align*}
\begin{pmatrix}
1 & -f_{0}(z)K_{D/N}
\end{pmatrix}
\begin{pmatrix}
\widehat{u}_{1}(z)
\\
\widehat{u}_{2}(z)
\end{pmatrix} =\widehat{u}_{1}(z)-(f_{0}(-x\partial _{x})K_{D/N}u_{2})
\,\widehat{}\,(z)=\widehat{v}(z)
\end{align*}
with
$v(x):=u_{1}(x)-f_{0}(-x\partial _{x})K_{D/N}u_{2}(x)\in {\mathscr{S}}^{\infty }(Y^{\wedge })$. Hence the range of
$\mathbf{G}^{(\ell )}_{\sigma }$ coincides with that of
$G^{(\ell )}_{\sigma ,D/N}$.
\end{proof}

\subsection{Extensions on the model cone}
\label{sec6.1}

Let $\lambda ^{D/N}_{j}$, $j\in \mathbb{N}_{0}$, be the eigenvalues of
the Dirichlet and Neumann Laplacian on $Y$ with repect to $h(0)$, respectively.
Assuming that $\mathbb{D}$ is connected and has non-empty boundary, we
have $\lambda ^{D}_{0}<0$ for the Dirichlet case while
$\lambda ^{N}_{0}=0$ for the Neumann case. Then $f_{0,D/N}(z)$ is invertible
for all $z$ except for the values
\begin{equation}
\label{eq:poles}
q_{j,D/N}^{\pm }:=\frac{n-1}{2}\pm \sqrt{\Big (\frac{n-1}{2}\Big )^{2}-
\lambda ^{D/N}_{j}}, \quad j\in \mathbb{N}_{0}.
\end{equation}
Note the relation $q_{j,D/N}^{-}+q_{j,D/N}^{+}=n-1$.

Let $E^{D/N}_{j}$ denote the eigenspace associated with
$\lambda ^{D/N}_{j}$ and $\pi _{j,D/N}$ the $L^{2}$-orthogonal projection
onto $E^{D/N}_{j}$. Then, in case $n\ge 2$,
\begin{equation*}
f_{0,D/N}^{-1}(z)=\sum _{j=0}^{\infty
}\frac{1}{q_{j,D/N}^{+}-q_{j,D/N}^{-}} \Big (\frac{1}{z-q_{j,D/N}^{+}}-
\frac{1}{z-q_{j,D/N}^{-}}\Big )\pi _{j,D/N},
\end{equation*}
hence
\begin{eqnarray}
\label{eq:Pi}
\Pi _{q_{j,D/N}^{\pm }}f_{0,D/N}^{-1}(z) =\pm
\frac{\pi _{j,D/N}}{q_{j,D/N}^{+}-q_{j,D/N}^{-}}\,(z-q_{j,D/N}^{\pm })^{-1}.
\end{eqnarray}
In case $n=1$ this holds also true in the Dirichlet case, and in the Neumann
case whenever $j\ge 1$. Moreover, in case $n=1$,
$q_{0,N}^{+}=q_{0,N}^{-}=0$ is a double pole and
\begin{equation*}
\Pi _{0}f_{0,N}^{-1}(z)=\pi _{0} z^{-2}.
\end{equation*}

\begin{definition}%
\label{eq:DN-spaces}
Let $E^{D/N}_{j}$ be the eigenspace associated with
$\lambda ^{D/N}_{j}$. Define
\begin{equation*}
\widehat{{\mathscr{E}}}_{q_{j,D/N}^{\pm }}=E^{D/N}_{j}\otimes x^{-q_{j,D/N}^{
\pm }} =\Big \{e(y)x^{-q_{j,D/N}^{\pm }}\mid e\in E^{D/N}_{j}\Big \},
\qquad j\in {\mathbb{N}}_{0}.
\end{equation*}
unless $n=1$, $j=0$ and we have Neumann boundary conditions; then we set
\begin{equation*}
\widehat{{\mathscr{E}}}_{0,N}=E^{N}_{0}+E^{N}_{0}\otimes \log x.
\end{equation*}
\end{definition}

For $\gamma \in \mathbb{R}$ define the set
\begin{equation}
\label{eq:igamma}
I_{\gamma ,D/N}=\Big \{q_{j,D/N}^{\pm }\mid j \in \mathbb{N}_{0}\Big \}
\cap \Big (\frac{n+1}{2}-\gamma -2,\frac{n+1}{2}-\gamma \Big ).
\end{equation}

By \reftext{Theorem~\ref{thm:Dmaxhat}}, \reftext{Corollary~\ref{cor:G_D/N}} and straight-forward
calculations using the residue theorem we obtain:

\begin{proposition}%
\label{prop:5.1}
The maximal extension of $\widehat{\Delta }_{D/N}$ in
${\mathcal{K}}^{0,\gamma }_{p}(Y^{\wedge })$, $\gamma \in {\mathbb{R}}$, subject
to Dirichlet/Neumann boundary conditions has the domain
\begin{equation*}
{\mathscr{D}}_{\max }(\widehat{\Delta }_{D/N})={\mathscr{D}}_{\min }(
\widehat{\Delta }_{D/N}) \oplus \omega \widehat{{\mathscr{E}}}_{D/N}^{
\gamma },\qquad \widehat{{\mathscr{E}}}^{\gamma }_{D/N}=\mathop{
\text{\Large $\oplus $}}_{q\in I_{\gamma ,D/N}} \widehat{{\mathscr{E}}}_{q,D/N}.
\end{equation*}
In case $\frac{n+1}{2}-\gamma -2\not =q^{\pm }_{j,D/N}$ for every $j$, the
minimal domain coincides with
${\mathcal{K}}^{2, \gamma +2}_{p} (Y^{\wedge })_{D/N}$.
\end{proposition}

The description of the adjoints of closed extensions makes use of the bilinear
form
\begin{equation}
\label{eq:pairing}
[u,v]_{D/N} :=(\widehat{\Delta }_{D/N}(\omega u),\omega v)_{{
\mathcal{K}}^{0,0}_{2}(Y^{\wedge })} -(\omega u,\widehat{\Delta }_{D/N}(
\omega v))_{{\mathcal{K}}^{0,0}_{2}(Y^{\wedge })},
\end{equation}
which is non-degenerate as a map
\begin{equation*}
\widehat{{\mathscr{E}}}^{\gamma }_{D/N}\times \widehat{{\mathscr{E}}}^{-
\gamma }_{D/N}\longrightarrow {\mathbb{C}};
\end{equation*}
it does not depend on the choice of the cut-off function $\omega $.

The result below, is an analog of \cite[Proposition 6.3]{CSS1}:

\begin{lemma}%
\label{lem:adjoint}
Let $\widehat{\underline{\Delta }}_{D/N}$ be an extension in
${\mathcal{K}}^{0,\gamma }_{p}(Y^{\wedge })$ with domain
\begin{equation*}
{\mathscr{D}}(\widehat{\underline{\Delta }}_{D/N})={\mathscr{D}}_{\min }(
\widehat{\Delta }_{D/N}) \oplus \omega
\underline{\widehat{{\mathscr{E}}}}_{D/N}^{\gamma }.
\end{equation*}
Then its adjoint, considered as an unbounded operator in
${\mathcal{K}}^{0,-\gamma }_{p^{\prime }}(Y^{\wedge })$, is the Laplacian acting
on the domain
\begin{equation*}
{\mathscr{D}}(\widehat{\underline{\Delta }}_{D/N})={\mathscr{D}}_{\min }(
\widehat{\Delta }_{D/N}) \oplus \omega
\underline{\widehat{{\mathscr{E}}}}_{D/N}^{\gamma ,\#},
\end{equation*}
where $\underline{\widehat{{\mathscr{E}}}}_{D/N}^{\gamma ,\#}$ is the orthogonal
space to $\underline{\widehat{{\mathscr{E}}}}^{\gamma }_{D/N}$ with respect
to the pairing \reftext{\eqref{eq:pairing}}.
\end{lemma}

\subsection{Extensions with property (E3)}
\label{sec6.2}

Of the three ellipticity conditions (E1), (E2) and (E3), generally the
last one is the most difficult to check. \reftext{Theorem~\ref{thm:5.3}}, below,
gives a simple sufficient condition. We focus on extensions
$\widehat{\underline{\Delta }}_{D/N}$ of $\widehat{\Delta }_{D/N}$ in
${\mathcal{K}}^{0,\gamma }_{2}(Y^{\wedge })$ with domain of the form
\begin{equation}
\label{eq:domain}
{\mathscr{D}}(\widehat{\underline{\Delta }}_{D/N}) = {\mathscr{D}}_{\min }(
\widehat{\Delta }_{D/N}) \oplus \omega
\underline{\widehat{{\mathscr{E}}}}_{D/N}^{\gamma },\qquad
\underline{\widehat{{\mathscr{E}}}}_{D/N}^{\gamma }=\mathop{
\text{\Large $\oplus $}}_{q\in I_{\gamma ,D/N}}
\underline{\widehat{{\mathscr{E}}}}_{q, D/N},
\end{equation}
where $\underline{\widehat{{\mathscr{E}}}}_{q,D/N}$ is an arbitrary subspace
of $\widehat{{\mathscr{E}}}_{q,D/N}$, except in case of the Neumann condition
and $n=1$, where for $q=0$ we confine ourselves to the following three
choices: $\underline{\widehat{{\mathscr{E}}}}_{0,N}=\{0\}$,
$\underline{\widehat{{\mathscr{E}}}}_{0,N}=E_{0}^{N}\otimes 1$, or
$\underline{\widehat{{\mathscr{E}}}}_{0,N}=\widehat{{\mathscr{E}}}_{0,N}$.

Let
$\underline{\widehat{{\mathscr{E}}}}_{q_{j}^{\pm }, D/N}=\underline{E}^{D/N}_{j}
\otimes x^{-q_{j,D/N}^{\pm }}$ with a subspace
$\underline{E}^{D/N}_{j}$ of $E^{D/N}_{j}$, cf. \reftext{Definition~\ref{eq:DN-spaces}}. We define
\begin{equation*}
\underline{\widehat{{\mathscr{E}}}}_{q_{j,D/N}^{\pm }}^{\perp }:=
\underline{E}^{D/N,\perp }_{j}\otimes x^{-q_{j,D/N}^{\mp }} \subseteq
\widehat{{\mathscr{E}}}_{q_{j,D/N}^{\mp }}
\end{equation*}
(note the sign change), where $\underline{E}^{D/N,\perp }_{j}$ is the orthogonal
complement in $E^{D/N}_{j}$ with respect to the $L^{2}(Y)$ inner product,
with the only exception for
$\underline{\widehat{{\mathscr{E}}}}_{0,N}=E_{0}^{N}\otimes 1$ in case
$n=1$, where instead we define
$\underline{\widehat{{\mathscr{E}}}}_{0,N}^{\perp }:=
\underline{\widehat{{\mathscr{E}}}}_{0,N}$.

\begin{lemma}%
\label{lem:orthogonal}
If $\underline{\widehat{{\mathscr{E}}}}_{D/N}^{\gamma }$ is as in \reftext{\eqref{eq:domain}} then
$\displaystyle \underline{\widehat{{\mathscr{E}}}}_{D/N}^{\gamma ,\#} =
\mathop{\text{\Large $\oplus $}}_{q\in I_{\gamma ,D/N}}
\underline{\widehat{{\mathscr{E}}}}_{q, D/N}^{\perp }$.
\end{lemma}
\begin{proof}
The result is based on the description of adjoint operators in
\cite[Section 6.3]{CSS1}. We shall focus on the Neumann case with
$n=1$; this is the most involved case, since then
$q_{0,N}^{+}=q_{0,N}^{-}=0$ is a double pole of the conormal symbol. The
other cases are treated analogously.

Let us write $q_{j}^{\pm }:=q_{j,N}^{\pm }$. Since $n=1$, we have
$I_{\gamma }:=I_{\gamma ,N} =(-1-\gamma ,1-\gamma )\cap \{q_{j}^{\pm
}\mid j\ge 0\}$ and $q_{j}^{-}=-q_{j}^{+}$ for all $j$. By symmetry it is
enough to consider the case $0\le \gamma <1$.

\medskip\noindent
\textbf{Step 1:} Let $e_{j},f_{j}\in E_{j}^{N}$, $e_{k}\in E_{k}^{N}$, be
arbitrary. A direct calculation yields
\begin{equation*}
[ e_{j} x^{-q_{j}^{\pm }}, e_{k} x^{-q_{k}^{\pm }}] ={-}(q_{j}^{\pm }-q_{k}^{\pm })(e_{j},e_{k})_{L^{2}(Y)} \int _{0}^{+\infty }\partial _{x}\omega ^{2}(x)
\,x^{-q_{k}^{\pm }-q_{j}^{\pm }}\,dx.
\end{equation*}
The second factor on the right-hand side equals zero whenever
$j\not =k$. In case $j=k$,
\begin{equation*}
[e_{j} x^{-q_{j}^{+}},f_{j} x^{-q_{j}^{+}}] =[e_{j} x^{-q_{j}^{-}},f_{j}
x^{-q_{j}^{-}}]=0
\end{equation*}
as well as
\begin{align*}
[e_{j} x^{-q_{j}^{+}},&f_{j} x^{-q_{j}^{-}}] =-
\overline{[e_{j} x^{-q_{j}^{-}},f_{j} x^{-q_{j}^{+}}]}
\\
& ={-}(q_{j}^{+}-q_{j}^{-})(e_{j},f_{j})_{L^{2}(Y)}\int _{0}^{+\infty }
\partial _{x}\omega ^{2}(x)\,dx =(q_{j}^{-}-q_{j}^{+})(e_{j},f_{j})_{L^{2}(Y)}.
\end{align*}
For $e,f\in E_{0}^{N}$, $e_{j}\in E_{j}^{N}$, $j>0$, one obtains
\begin{equation*}
[e,f]= [e\log x,f\log x] =[e,e_{j} x^{-q_{j}^{\pm }}]=[e\log x,e_{j} x^{-q_{j}^{\pm }}]=0
\end{equation*}
and
\begin{equation*}
[e,f\log x]=-(e,f)_{L^{2}(Y)}\int _{0}^{+\infty }\partial _{x}\omega ^{2}(x)
\, dx =(e,f)_{L^{2}(Y)}.
\end{equation*}
\textbf{Step 2:} Let $q\in I_{\gamma }\setminus I_{-\gamma }$. Then
$q=q_{j}^{-}$ for some $j>0$. From the above calculations for the pairing
it follows that
\begin{equation*}
\underline{\widehat{{\mathscr{E}}}}_{q}^{\#}=
\underline{\widehat{{\mathscr{E}}}}_{q}^{\perp }\oplus \mathop{
\text{\Large $\oplus $}}_{\substack{p\in I_{-\gamma },\\ p\neq-q}}
\widehat{{\mathscr{E}}}_{p}.
\end{equation*}
Consequently, as the orthogonal complement of a sum of spaces is the intersection
of all respective orthogonal complements, we obtain
\begin{align}
\label{eq:compl01}
\Big (\mathop{\text{\Large $\oplus $}}_{q\in I_{\gamma }\setminus I_{-
\gamma }}\underline{\widehat{{\mathscr{E}}}}_{q}\Big )^{\#} =\mathop{
\text{\Large $\oplus $}}_{\substack{p\in I_{\gamma }\cap I_{-\gamma }}}
\widehat{{\mathscr{E}}}_{p} \oplus \mathop{\text{\Large $\oplus $}}_{q
\in I_{\gamma }\setminus I_{-\gamma }}
\underline{\widehat{{\mathscr{E}}}}_{q}^{\perp }.
\end{align}
Now let $q\in I_{\gamma }\cap I_{-\gamma }$. Then also
$-q\in I_{\gamma }\cap I_{-\gamma }$. If $q\neq0$, the above calculations
yield that
\begin{align*}
(\underline{\widehat{{\mathscr{E}}}}_{q}\oplus
\underline{\widehat{{\mathscr{E}}}}_{-q})^{\#} &=
\underline{\widehat{{\mathscr{E}}}}_{q}^{\perp }\oplus
\underline{\widehat{{\mathscr{E}}}}_{-q}^{\perp }\mathop{
\text{\Large $\oplus $}}_{
\substack{p\in I_{-\gamma },\\ p\notin \{-q,q\}}}
\widehat{{\mathscr{E}}}_{p},
\end{align*}
while for $q=0$, due to our choices of
$\underline{\widehat{{\mathscr{E}}}}_{0}$, we find
\begin{equation*}
\underline{\widehat{{\mathscr{E}}}}_{0}^{\#} =
\underline{\widehat{{\mathscr{E}}}}_{0}^{\perp }\oplus \mathop{
\text{\Large $\oplus $}}_{\substack{p\in I_{-\gamma },\\ p\neq0}}
\widehat{{\mathscr{E}}}_{p}.
\end{equation*}
It follows that
\begin{align}
\label{eq:compl02}
\Big (\underline{\widehat{{\mathscr{E}}}}_{0}\oplus \mathop{
\text{\Large $\oplus $}}_{
\substack{q\in I_{\gamma }\cap I_{-\gamma },\\q<0}}
\underline{\widehat{{\mathscr{E}}}}_{q}\oplus
\underline{\widehat{{\mathscr{E}}}}_{-q}\Big )^{\#} =
\underline{\widehat{{\mathscr{E}}}}_{0}^{\perp }\oplus \mathop{
\text{\Large $\oplus $}}_{
\substack{q\in I_{\gamma }\cap I_{-\gamma },\\q<0}}
\underline{\widehat{{\mathscr{E}}}}_{q}^{\perp }\oplus
\underline{\widehat{{\mathscr{E}}}}_{-q}^{\perp }\oplus \mathop{
\text{\Large $\oplus $}}_{p\in I_{-\gamma }\setminus I_{\gamma }}
\widehat{{\mathscr{E}}}_{p}.
\end{align}
Taking the intersection of \reftext{\eqref{eq:compl01}} and \reftext{\eqref{eq:compl02}} yields
the claim.
\end{proof}

\begin{theorem}%
\label{thm:5.3}
Let $|\gamma | <(n+1)/2$, and suppose that the
$\revtex q_{j,D/N}^{\pm }$ are different from both
$\frac{n+1}{2}-\gamma $ and $\frac{n+1}{2}-\gamma -2$ for all $j$. Moreover,
let $\underline{\widehat{\Delta }}_{D/N}$ be an extension with domain as in \reftext{\eqref{eq:domain}}, where the spaces
$\underline{\widehat{{\mathscr{E}}}}_{q,D/N}$ are chosen such that:
\begin{enumerate}
\item[(1)]
$\underline{\widehat{{\mathscr{E}}}}_{q,D/N}^{\perp }=
\underline{\widehat{{\mathscr{E}}}}_{n-1-q,D/N}$ for
$\displaystyle q\in I_{\gamma }\cap I_{-\gamma }$,
\item[(2)]
$\underline{\widehat{{\mathscr{E}}}}_{q,D/N} = {\widehat{{\mathscr{E}}}}_{q,D/N}$
for $\gamma \ge 0$ and $q\in I_{\gamma }\setminus I_{-\gamma }$,
\item[(3)] $\underline{\widehat{{\mathscr{E}}}}_{q,D/N} = \{0\}$ for
$\gamma \le 0$ and $q\in I_{\gamma }\setminus I_{-\gamma }$.
\end{enumerate}
Then $\underline{\widehat{\Delta }}_{D/N}$ satisfies \textup{(E3) for every sector
$\Lambda \subseteq {\mathbb{C}}\setminus {\mathbb{R}}_{+}$}.
\end{theorem}

\begin{proof}
All extensions of the form \reftext{\eqref{eq:domain}} are invariant under dilations
in the sense of Section~\ref{subsec:dilation}. The decay condition in (E3)
therefore follows via homogeneity, provided we can establish the invertibility
of
$\eta ^{\mu }-x^{-\gamma } \underline{\widehat{\Delta }}_{D/N}x^{\gamma }$ in
${\mathcal{K}}^{0,0}_{2}(Y^{\wedge })$ for $|\eta |=1$,
$\eta \in {\mathbb{C}}\setminus \overline{{\mathbb{R}}}_{+}$. This in turn
is equivalent to the invertibility of
$\eta ^{\mu }-\underline{\widehat{\Delta }}_{D/N}$ in
${\mathcal{K}}^{0,\gamma }_{2}(Y^{\wedge })^{-\gamma }$. Since both the Dirichlet
and the Neumann boundary condition commute with cut-off functions, \reftext{Proposition~\ref{prop:weight}} shows that it suffices to establish the invertibility
of $\eta ^{\mu }-\underline{\widehat{\Delta }}_{D/N}$ in
${\mathcal{K}}^{0,\gamma }_{2}(Y^{\wedge })$.

As observed in the proof of \reftext{Proposition~\ref{prop:SI_for_model_cone}},
$\eta ^{\mu }-\underline{\widehat{\Delta }}_{D/N}$ is a Fredholm operator.
Moreover, we may assume $\gamma \ge 0$ by possibly going over to the adjoint
problem, which satisfies the conditions (1) and (2) above by
\cite[Theorem 6.3]{CSS1}. Then we argue in the same way as for Theorem
5.7 in \cite{SS}.
\end{proof}

\subsection{An extension of the Neumann Laplacian}
\label{sec6.3}

With a view towards an application discussed below, we will study a particular
extension of the Neumann Laplacian. We recall that $q_{1,N}^{-}<0$ and,
as in \cite{RS1,RS,RS3,RS4,SS2}, we fix $\gamma $ with
\begin{align}
\label{gamma}
\gamma =\frac{n-3}{2}+\delta \text{ with some }0<\delta <\min \{-q_{1,N}^{-},2
\}, \ 2-\delta \neq q_{j,N}^{+} \text{ for all }j.
\end{align}
Then $I_{\gamma ,N}$ contains $q_{0,N}^{-}=0$, but none of the
$q_{j,N}^{-}$ for $j>0$, whereas $I_{-\gamma ,N}$ contains
$q_{0,N}^{+}$, but none of the $q^{+}_{j,N}$ for $j>0$. Moreover,
$I_{\gamma ,N}\cap I_{-\gamma ,N}$ contains at most the poles
$q^{-}_{0,N}=0$ and $q^{+}_{0,N}=n-1$. In fact, for $n=1$,
$I_{\gamma ,N}\cap I_{-\gamma ,N}$ contains only
$q_{0,N}^{+}=q_{0,N}^{-}=0$. For $n=2$, the intersection consists of both
$q_{0,N}^{-}=0$ and $q_{0,N}^{+}=1$, provided $\delta <1$, else it is empty.
For $n\ge 3$, the intersection is always empty.

Let us determine the space
${\mathscr{E}}_{0,N}=\theta _{0}^{-1}\widehat{\mathscr{E}}_{0,N}$. The computation
extends that in Section 6.4.1 of \cite{SS2} to the Neumann Laplacian.

We need the conormal symbol $\sigma ^{1}_{M}({\mathcal{A}}_{N})$. Recall
the unit vector field $\nu _{x}$ defined at the begining of this section.
For every $y$ in a collar-neighborhood of the boundary $\partial Y$,
$x\mapsto \nu _{x}(y)$ represents a smooth curve in $T_{y}Y$; let
\begin{equation*}
\nu ^{\prime }_{x}(y)=\frac{d}{dx}\,\nu _{x}(y)\in T_{y}Y.
\end{equation*}
This defines a vector-field $\nu ^{\prime }_{x}$ for each $x\ge 0$. Then
we have
\begin{equation*}
\sigma ^{1}_{M}({\mathcal{A}}_{N})(z)=
\begin{pmatrix}
\Delta _{Y}^{\prime }(0)-H^{\prime }(0)z
\\
\gamma _{1,N}^{\prime
}\end{pmatrix}
, \qquad \gamma ^{\prime }_{1}:=\gamma _{D} \partial _{\nu ^{\prime }_{0}}.
\end{equation*}
Note that locally constant functions on $Y$ belong both to the kernel of
$\Delta _{Y}^{\prime }(0)$ and the kernel of $\gamma ^{\prime }_{1,N}$, i.e.,
$\Delta _{Y}^{\prime }(0)\pi _{0,N}=0$ and
$\gamma ^{\prime }_{1,N}\pi _{0,N}=0$.

\medskip\noindent
\textbf{The case $n\ge 2$:} If $\delta <1$, then
$G_{0,N}=G_{0,N}^{(0)}$ by definition. If $\delta >1$, the origin is the
only pole of $f_{0,N}(z)^{-1}$ in $I_{\gamma }=(-\delta ,2-\delta )$ and \reftext{\eqref{eq:Pi}} implies that
\begin{align*}
\mathbf{g}_{1,N}(z)\Pi _{0}(f_{0,N}(z)^{-1}\widehat{v}(z)) &=-
\mathbf{f}_{0,N}(z-1)^{-1}
\begin{pmatrix}
\Delta _{Y}^{\prime }(0)-H^{\prime }(0) z
\\
\gamma _{1,N}^{\prime
}\end{pmatrix} \frac{\pi _{0,N} \widehat{v}(0)}{(n-1)z}
\\
&=f_{0,N}(z-1)^{-1}H^{\prime }(0)\pi _{0,N} \widehat{v}(0)/(n-1)
\end{align*}
is holomorphic in $z=0$. Hence $G_{0,N}^{(1)}=0$ and
$G_{0,N}=G_{0,N}^{(0)}$. We conclude that, for every choice of
$0<\delta <\min \{-q_{1,N}^{-},2\}$,
\begin{equation*}
{\mathscr{E}}_{0,N}=\widehat{\mathscr{E}}_{0,N}=E_{0}^{N}\otimes 1.
\end{equation*}
\textbf{The case $n=1$:} By direct calculation
\begin{align*}
(G_{0,N}^{(0)} v)(x) = \int _{|z|<\varepsilon } x^{-z}\pi _{0,N}\left (
\frac{\widehat{v}(0)}{z^{2}}+\frac{\widehat{v}^{\prime }(0)}{z}\right )\,
\dbar z =\log x \ \pi _{0,N}\widehat{v}(0) +\pi _{0,N}\widehat{v}^{\prime }(0),
\end{align*}
showing that
\begin{equation*}
\widehat{{\mathscr{E}}}_{0,N} = \big \{e_{0} + e_{1}\log x\mid e_{0},e_{1}
\in E_{0}^{N}\big \}=E_{0}^{N}\otimes 1+ E_{0}^{N}\otimes \log x.
\end{equation*}
By definition, $G_{0,N}=G_{0,N}^{(0)}$ for $\delta <1$, while for
$\delta \ge 1$, similarly as before,
\begin{align*}
\mathbf{g}_{1,N}(z)(\Pi _{0}f_{0,N}^{-1}\widehat{v})(z) =&f_{0,N}(z-1)^{-1}H^{\prime }(0)\pi _{0,N}\left (\frac{\widehat{v}(0)}{z}+\widehat{v}^{\prime }(0)
\right ).
\end{align*}
Therefore
\begin{equation*}
(G_{0,N}^{(1)}v)(x) = x a_{N} \pi _{0,N}\widehat{v}(0),\qquad a_{N}:=(1+
\Delta _{Y,N}(0))^{-1}H^{\prime }(0) \in {\mathscr{C}}^{\infty }(Y).
\end{equation*}
We conclude that
\begin{align*}
{\mathscr{E}}_{0,N} =
\begin{cases}
\widehat{{\mathscr{E}}}_{0,N},& \delta <1
\\[1mm]
\big \{e_{0} + e_{1}(\log x +x a_{N})\mid e_{0},e_{1}\in E_{0}^{N}
\big \}, &\delta \ge 1
\end{cases}.
\end{align*}
The isomorphism
$\theta _{0}:{\mathscr{E}}_{0,N}\to \widehat{{\mathscr{E}}}_{0,N}$ is the
identity map in case $\delta <1$, otherwise
\begin{equation*}
\theta _{0}(e_{0} + e_{1}(\log x +x a_{N}))=e_{0} + e_{1}\log x,
\qquad e_{0},e_{1}\in E_{0}^{N}.
\end{equation*}
As a result of this computation we obtain:
\begin{corollary}%
\label{cor:E}
If $\underline{{\mathscr{E}}}_{0,N}=E_{0}^{N}\otimes 1$, then
$\underline{\widehat{{\mathscr{E}}}}_{0,N}=E_{0}^{N}\otimes 1$.
\end{corollary}

\begin{theorem}%
\label{thm:Ndomain}
Let $\gamma $ be as in \reftext{\eqref{gamma}} and $1<p<+\infty $. If
$\underline{\Delta }_{N}$ denotes the extension of the Neumann Laplacian
with domain
\begin{equation}
\label{domain}
{\mathscr{D}}(\underline{\Delta }_{N}) = {\mathcal{H}}^{2,\gamma +2}_{p}({
\mathbb{D}})_{N}\oplus \omega \underline{{\mathscr{E}}}_{0,N},
\end{equation}
then $c-\underline{\Delta }_{N}$ has a bounded $H_{\infty }$-calculus in
${\mathcal{H}}^{0,\gamma }_{p}({\mathbb{D}})$ for sufficiently large
$c>0$.
\end{theorem}

\begin{proof}
Clearly, the Neumann Laplacian satisfies condition (E1). By our choice
of $\gamma $, also (E2) holds. Finally, \reftext{Theorem~\ref{thm:5.3}} in connection
with \reftext{Corollary~\ref{cor:E}} implies condition (E3). Hence the assertion
follows from \reftext{Theorem~\ref{thm:main}}.
\end{proof}

\section{The porous medium equation on conic manifolds with boundary}
\label{sec:PME}

Following up on the investigations in \cite{RS3,RS4} and \cite{SS2} for
the case of conic manifolds without boundary, we shall show how the above
results can be applied to the porous medium equation
\begin{align*}
u'(t)-\Delta u^{m}(t) &= f(t,u) &&\text{in } {\mathbb{D}}\text{ for } t
\in (0,T),
\\
\gamma _{N} u(t) &= 0&& \text{on }\partial {\mathbb{D}}\text{ for } t
\in (0,T),
\\
u(0) &= u_{0} && \text{in } {\mathbb{D}},
\end{align*}
where $m>0$, $T>0$, $f$ is a forcing term and $u_{0}$ is some given initial
datum.

As long as $u$ is strictly positive, we can make the transformation
$u=v^{m}$ and obtain the equivalent system
\begin{align}
\label{PMEa1}
v'(t)-mv^{(m-1)/m}\Delta v (t) &= g(t,v) &&\text{in } {\mathbb{D}}
\text{ for } t\in (0,T),
\eqncr
\label{PMEa2}
\gamma _{N} v(t) &= 0&& \text{on }\partial {\mathbb{D}}\text{ for } t
\in (0,T),
\eqncr
\label{PMEa3}
v(0) &= v_{0} && \text{in } {\mathbb{D}},
\end{align}
with $g(t,v) = f(t, v^{1/m})$. In the sequel we will assume that $g$ is
holomorphic in $v$ and Lipschitz in $t$.

Equation \reftext{\eqref{PMEa1}} is a quasilinear evolution equation to which we
will apply the following theorem of Cl\'{e}ment and Li.

\begin{theorem}%
\label{CL}
Consider the quasilinear evolution equation
\begin{align*}
v'(t) +A(v(t)) v(t) =g(t,v), \qquad v(0)=v_{0}.
\end{align*}
Let $X_{0}$ and $X_{1}$ be Banach spaces and $V$ an open neighborhood of
$v_{0}$ in the real interpolation space
$X_{1-1/q,q}= (X_{0},X_{1})_{1-1/q,q}$ such that
$A(v_{0}): X_{1}\rightarrow X_{0}$ has maximal $L^{q}$-regularity and that,
for some $T_{0}>0$,
\begin{itemize}
\item[(H1)] $A\in C^{1-}(V, {\mathscr{L}}(X_{1},X_{0}))$,
\item[(H2)] $f\in C^{1-,1-}([0,T_{0}]\times V, X_{0})$.
\end{itemize}
Then there exists a $T\in (0,T_{0}]$ and a unique solution
$v\in L_{q}(0,T;X_{1})\cap W^{1}_{q}(0,T;X_{0})$ on $(0,T)$. In particular,
$v\in C([0,T];X_{1-1/q,q})$ by \cite[Theorem III.4.10.2]{Am}.
\end{theorem}

A central property is the maximal $L^{q}$-regularity of the operator
$A(v_{0})$. We recall that all the Mellin-Sobolev spaces used here are
UMD Banach spaces and therefore the existence of a bounded
$H_{\infty }$-calculus implies the ${\mathcal{R}}$-sectoriality for the same
sector according to Cl\'{e}ment and Pr\"{u}ss, \cite[Theorem 4]{CP}. Moreover,
every operator, which is ${\mathcal{R}}$-sectorial on
$\Lambda (\theta )$ for $\theta <\pi /2$, has maximal $L^{q}$-regularity,
$1<q<+\infty $, see Weis \cite[Theorem 4.2]{W}.

For $\gamma $ and $\delta $ as in \reftext{\eqref{gamma}} we fix
$1<p,q<+\infty $ such that
\begin{equation}
\label{pq}
\frac{n+1}{p}+\frac{2}{q}<1 \text{ and } \frac{2}{q}<\delta .
\end{equation}
We shall apply the theorem of Cl\'{e}ment and Li with
$A_{c}(v)= c-mv^{(m-1)/m}\underline{\Delta }_{N}$, where
$\underline{\Delta }_{N}$ is the realization of the Neumann Laplacian with
the domain in \reftext{\eqref{domain}}, and the Banach spaces
$X_{0} = {\mathcal{H}}^{0,\gamma }_{p}({\mathbb{D}}) $ and
$X_{1} = {\mathcal{H}}^{2,\gamma +2}_{p}({\mathbb{D}})_{N} \oplus
\omega {\mathscr{E}}_{0,N} = {\mathscr{D}}(\underline{\Delta }_{N})$.

\subsection{Interpolation spaces}
\label{sec7.1}

The following observation will be useful in the sequel.

\begin{lemma}%
\label{interpolation}
Let $X_{0},X_{1}$ and $U$ be Banach spaces, all continuously embedded in
the same Hausdorff topological vector space. Assume that $U$ has finite
dimension. Then
\begin{equation}
\label{eq:inter}
(X_{0}+U,X_{1}+U)_{\theta ,q}=(X_{0},X_{1})_{\theta ,q}+U
\end{equation}
for every choice of $1<q<+\infty $ and $0<\theta <1$.
\end{lemma}
\begin{proof}
Clearly the right-hand side in \reftext{\eqref{eq:inter}} is continuously embedded
in the left-hand side. By the inverse mapping theorem, it remains to show
that the left-hand side is a subset of the right-hand side.

Given $x\in (X_{0}+U,X_{1}+U)_{\theta ,q}$, let $\widetilde{X}_{j}$ be
a topological complement of $U\cap X_{j}$ in $X_{j}$. Then
$X_{j}+U=\widetilde{X}_{j}\oplus U$ with equivalent norms. Write
$x=\widetilde{x}+u$ with
$\widetilde{x}\in \widetilde{X}_{0}+\widetilde{X}_{1}$ and $u\in U$. Since
the norms are equivalent, there exists a $C\ge 0$ such that
\begin{align*}
\|\widetilde{x}_{0}\|_{X_{0}}+t\|\widetilde{x}_{1}\|_{X_{1}} &\le \|
\widetilde{x}_{0}\|_{X_{0}}+t\|\widetilde{x}_{1}\|_{X_{1}}+ \|u_{0}\|+t
\|u_{1}\|
\\
&\le C\big (\|\widetilde{x}_{0}+u_{0}\|_{X_{0}+U}+t\|\widetilde{x}_{1}+u_{1}
\|_{X_{1}+U}\big )
\end{align*}
for every $t>0$, whenever
$\widetilde{x}=\widetilde{x}_{0}+\widetilde{x}_{1}$ with
$\widetilde{x}_{j}\in \widetilde{X}_{j}$ and $u=u_{0}+u_{1}$ with
$u_{j}\in U$. By passing to the infimum over all such representations we
find
\begin{equation*}
K(t,\widetilde{x};\widetilde{X}_{0},\widetilde{X}_{1})\le C K(t,x;X_{0}+U,X_{1}+U),
\end{equation*}
where $K(\cdot )$ is the usual $K$-functional in the definition of the
real interpolation method. It follows that
$\widetilde{x}\in (\widetilde{X}_{0},\widetilde{X}_{1})_{\theta ,q}
\hookrightarrow (X_{0},X_{1})_{\theta ,q}$.
\end{proof}

\begin{lemma}%
\label{realint}
Let $s_{1},s_{0},\gamma _{1},\gamma _{0}\in \mathbb{R}$,
$1<p,q<+\infty $, and $0<\theta <1$ be arbitrary. Then
\begin{equation}
\label{realint2}
{\mathcal{H}}^{s+\varepsilon ,\gamma +\varepsilon }_{p}({\mathbb{D}})
\hookrightarrow ({\mathcal{H}}^{s_{0},\gamma _{0}}_{p}({\mathbb{D}}),{
\mathcal{H}}^{s_{1},\gamma _{1}}_{p}({\mathbb{D}}))_{\theta ,q}
\hookrightarrow {\mathcal{H}}^{s-\varepsilon ,\gamma -\varepsilon }_{p}({
\mathbb{D}})
\end{equation}
for all $\varepsilon >0$, where $s=(1-\theta )s_{0}+\theta s_{1}$ and
$\gamma =(1-\theta )\gamma _{0}+\theta \gamma _{1}$.
\end{lemma}

\begin{proof}
Let $2{\mathbb{D}}$ be the smooth manifold with boundary obtained by gluing
two copies of ${\mathbb{D}}$ along $\{0\}\times Y$. It is then well-known
that
\begin{equation*}
H^{s+\varepsilon }_{p}(2{\mathbb{D}})\hookrightarrow (H^{s_{0}}_{p}(2{
\mathbb{D}}),H^{s_{1}}_{p}(2{\mathbb{D}}))_{\theta ,p} \hookrightarrow H^{s-
\varepsilon }_{p}(2{\mathbb{D}}).
\end{equation*}
Let $2Y$ denote a closed manifold containing $Y$. Proceeding as in the
proof of \cite[Lemma 5.4]{CSS2} and using duality, one finds that
\begin{equation*}
{\mathcal{H}}^{s+\varepsilon ,\gamma +\varepsilon }_{p}({\mathbb{R}}_{+}
\times 2Y)\hookrightarrow ({\mathcal{H}}^{s_{0},\gamma _{0}}_{p}({
\mathbb{R}}_{+}\times 2Y),{\mathcal{H}}^{s_{1},\gamma _{1}}_{p}({
\mathbb{R}}\times 2Y))_{\theta ,p}\hookrightarrow {\mathcal{H}}^{s-
\varepsilon ,\gamma -\varepsilon }_{p}({\mathbb{R}}_{+}\times 2Y).
\end{equation*}
With the help of a continuous extension operator as well as the restriction
operator, one finds the latter embeddings also for the spaces on
${\mathbb{R}}\times Y$. By a standard partition of unity argument we obtain \reftext{\eqref{realint2}} in case $p=q$. The general case follows from the embedding
results for interpolation spaces:
\begin{equation*}
(X_{0},X_{1})_{\theta ,q}\hookrightarrow (X_{0},X_{1})_{\theta ',p}
\hookrightarrow (X_{0},X_{1})_{\theta '',q}, \quad \theta ''< \theta '<
\theta , 1<p,q<+\infty ,
\end{equation*}
see \cite[(I.2.5.2)]{Am}.
\end{proof}

\begin{proposition}%
\label{internewman}
Let $\gamma $ and
${\mathscr{D}}(\underline{\Delta }_{N})={\mathcal{H}}^{2,\gamma +2}_{p}({
\mathbb{D}})_{N}\oplus \omega \underline{{\mathscr{E}}}_{0,N}$ be as in
\reftext{Theorem~\textup{\ref{thm:Ndomain}}} with $1<p<+\infty $ fixed. Let
$1<q<+\infty $ and $\theta \in (0,1)$ with $2\theta >1+\frac{1}{p}$. Then,
for every $\varepsilon >0$,
\begin{align*}
({\mathcal{H}}^{0,\gamma }_{p}({\mathbb{D}}),{\mathscr{D}}(
\underline{\Delta }_{N}))_{\theta ,q} \hookrightarrow {\mathcal{H}}^{2
\theta -\varepsilon ,\gamma +2\theta -\varepsilon }_{p}({\mathbb{D}})_{N}
\oplus \omega \underline{{\mathscr{E}}}_{0,N}.
\end{align*}
\end{proposition}
\begin{proof}
Note that
$\omega \underline{{\mathscr{E}}}_{0,N}\subset \mathcal{H}^{0,\gamma }_{p}(
\mathbb{D})$ by the choice of $\gamma $. By \reftext{Lemmas~\ref{interpolation} and \ref{realint}}, it remains to prove that
\begin{align*}
({\mathcal{H}}^{0,\gamma }_{p}({\mathbb{D}}),{\mathcal{H}}^{2,\gamma +2}_{p}({
\mathbb{D}})_{N})_{\theta ,q} \subset \mathrm{ker}\,\gamma _{N}.
\end{align*}
If $\Lambda $ is any fixed sector that does not contain the positive reals,
there exists a $c>0$ such that $c-\underline{\Delta }_{N}$ is sectorial
with respect to $\Lambda $, cf. \reftext{Theorem~\ref{thm:Ndomain}}. Set
$R_{N}(\lambda ):=(\lambda +c- \underline{\Delta }_{N})^{-1}$. Then
$R_{N}(\lambda )$ is uniformly bounded in
${\mathscr{L}}({\mathcal{H}}^{0,\gamma }_{p}({\mathbb{D}}),{\mathcal{H}}^{2,
\gamma }_{p}({\mathbb{D}}))$ and $(1+|\lambda |)R_{N}(\lambda )$ is uniformly
bounded in ${\mathscr{L}}({\mathcal{H}}^{0,\gamma }_{p}({\mathbb{D}}))$. From
complex interpolation and the fact that
$[{\mathcal{H}}^{s_{0},\gamma }_{p}({\mathbb{D}}),{\mathcal{H}}^{s_{1},
\gamma }_{p}({\mathbb{D}})]_{\rho }={\mathcal{H}}^{s,\gamma }_{p}({
\mathbb{D}})$ with $s=(1-\rho )s_{0}+\rho s_{1}$, we then obtain that
\begin{equation*}
\|R_{N}(\lambda )\|_{{\mathscr{L}}({\mathcal{H}}^{0,\gamma }_{p}({
\mathbb{D}}),{\mathcal{H}}^{2\rho ,\gamma }_{p}({\mathbb{D}}))} \le C(1+|
\lambda |)^{-1+\rho }, \quad 0<\rho <1.
\end{equation*}
Hence, for $\theta >\rho $,
\begin{equation*}
(c-\underline{\Delta }_{N})^{-\theta }=\frac{\sin \pi \theta }{\pi }\int _{0}^{\infty }s^{-\theta }R_{N}(s)\,ds
\end{equation*}
with the integral converging in
${\mathscr{L}}({\mathcal{H}}^{0,\gamma }_{p}({\mathbb{D}}),{\mathcal{H}}^{2
\rho ,\gamma }_{p}({\mathbb{D}}))$. For $2\rho >1+{1}/{p}$ we find that
$\gamma _{N}(c-\underline{\Delta }_{N})^{-\theta }=0$, since
$\gamma _{N}\in {\mathscr{L}}({\mathcal{H}}^{2\rho ,\gamma }_{p}({
\mathbb{D}}),{\mathcal{H}}^{0,\gamma -1/2}_{p}({\mathbb{B}}))$ can be pulled
under the integral and $\gamma _{N}R_{N}(s)\equiv 0$. Hence
$\gamma _{N}$ vanishes on
${\mathscr{D}}((c-\underline{\Delta }_{N})^{\theta })$ and the assertion follows
from the embedding
$({\mathcal{H}}^{0,\gamma }_{p}({\mathbb{D}}),{\mathcal{H}}^{2,\gamma +2}_{p}({
\mathbb{D}})_{N})_{\theta ,q} \hookrightarrow {\mathscr{D}}((c-
\underline{\Delta }_{N})^{\theta '})$ for $\theta '<\theta $, see
\cite[(I.2.9.6)]{Am}.
\end{proof}

\subsection{Solving the porous medium equation}
\label{sec7.2}

\begin{proposition}%
\label{rbound}
Let $v\in (X_{0},X_{1})_{1-1/q,q}$ with $v\ge \alpha >0$. For each
$\theta \in (0,\pi )$ there exists a $c>0$ such that
$A_{c}(v) = c-mv^{(m-1)/m}\underline{\Delta }_{N}$ is ${\mathcal{R}}$-sectorial
of angle $\theta $. In particular, $A(v)$ has maximal $L^{q}$-regularity.
\end{proposition}

This follows with minor modifications as in case without boundary, see
the proof of Theorem 6.1 in \cite{RS3}.

\begin{remark}
Combining the proof of \cite[Theorem 6.1]{RS3} with the method used in
the proof of \cite[Theorem 5.7]{DHP}, we obtain the above result even for
the case of a continuous function $v$ on ${\mathbb{D}}$ which is bounded
away from zero.
\end{remark}

\begin{theorem}%
Choose $\gamma $, $p$ and $q$ as in \reftext{\eqref{gamma}} and \reftext{\eqref{pq}}. Then
the porous medium equation \reftext{\eqref{PMEa1}}, \reftext{\eqref{PMEa2}},
\reftext{\eqref{PMEa3}} has a unique short time solution
\begin{equation*}
v\in W^{1,q}(0,T;\mathcal{H}^{0,\gamma }_{p}(\mathbb{D}))\cap L^{q}(0,T;
\mathcal{H}^{2,\gamma +2}_{p}(\mathbb{D})_{{N}}\oplus \omega {
\mathscr{E}}_{0,N})
\end{equation*}
for every strictly positive initial datum
$v_{0}\in (X_{0},X_{1})_{1-1/q,q}$.

In particular, $v\in C([0,T];(X_{0},X_{1})_{1-1/q,q})$ which, according
to \reftext{Proposition~\textup{\ref{internewman}}}, embeds into
$C([0,T];\mathcal{H}^{2-\frac{2}{q}-\varepsilon ,\gamma +2-
\frac{2}{q}-\varepsilon }_{p} (\mathbb{D})_{N}\oplus \omega {
\mathscr{E}}_{0,N})$ for every $\varepsilon >0$. If $g$ is independent of
$t$, then we additionally have
$v\in C^{\infty }((0,T);\mathcal{H}^{2,\gamma +2}_{p}(\mathbb{D})_{{N}}
\oplus \omega {\mathscr{E}}_{0,N})$.
\end{theorem}

\begin{proof}
According to \reftext{Proposition~\ref{rbound}}, the operator $A_{c}(v_{0})$ has
maximal regularity. Choose a neighborhood $V$ of $v_{0}$ in
$(X_{0},X_{1})_{1-1/q,q}$ such that $0<c_{1}\le \revtex v\le c_{2}$ for
all $v\in V$ and positive constants $c_{1},c_{2}$. Since the interpolation
space embeds into $C({\mathbb{D}})$, the space of continuous functions on
the compact space ${\mathbb{D}}$, the mapping $v\mapsto mv^{(m-1)/m}$ is
a smooth map from $V$ to $C({\mathbb{D}})$; its range consists of functions
with real part bounded and bounded away from zero. In particular,
\begin{equation}
\label{smooth}
v\mapsto A(v)\in C^{\infty }(V; {\mathscr{L}}(X_{1},X_{0})),
\end{equation}
so that condition (H1) is fulfilled. As (H2) holds by assumption, the theorem
of Cl\'{e}ment and Li shows the existence of some $T>0$ and a unique
\begin{equation*}
v\in W^{1,q}(0,T;\mathcal{H}^{0,\gamma }_{p}(\mathbb{D}))\cap L^{q}(0,T;
\mathcal{H}^{2,\gamma +2}_{p}(\mathbb{D})_{{N}}\oplus \omega {
\mathscr{E}}_{0,N})
\end{equation*}
solving \reftext{\eqref{PMEa1}}-\reftext{\eqref{PMEa3}}. In particular,
$v\in C([0,T]; (X_{0},X_{1})_{1-1/q,q})$. If $g$ is independent of
$t$, \reftext{\eqref{smooth}} together with \cite[Theorem 5.2.1]{PS} implies that
$v\in C^{\infty }((0,T);X_{1})$.
\end{proof}

\begin{remark}
As $v$ is continuous, bounded and strictly positive, $u=v^{1/m}$ furnishes
a solution
$u\in W^{1,q}(0,T;\mathcal{H}^{0,\gamma }_{p}(\mathbb{D}))\cap L^{q}(0,T;
\mathcal{H}^{2,\gamma +2}_{p}(\mathbb{D})_{{N}}\oplus \omega {
\mathscr{E}}_{0,N}) $ to the porous medium equation in the original form,
see \cite[Remark 2.12]{RS4}.
\end{remark}

\section{Appendix}
\label{sec8}

By $\omega , \omega _{1}, \omega _{2}$ we denote cut-off functions near
$x=0$, i.e. smooth non-negative functions on ${\mathbb{D}}$, supported in
$[0,1)\times Y$, equal to $1$ for small $x$. For simplicity of the presentation
we shall mostly omit the reference to the vector bundles.

\subsection{Function spaces on conic manifolds with boundary}
\label{sec8.1}

We briefly recall the definition of the function spaces used in this article.
More details can be found for example in \cite{KS} or \cite{HS}.

We denote by $2Y$ the double of $Y$; it is a closed manifold. Then
\begin{align}
\label{eq:SobZyl}
\begin{split}
H^{s}_{p}({\mathbb{R}}\times Y)&:=H^{s}_{p}({\mathbb{R}}\times 2Y)\big |_{{
\mathbb{R}}\times Y},
\\
\Hcirc ^{s}_{p}({\mathbb{R}}\times Y) &:=\left \{  u\in H^{s}_{p}({
\mathbb{R}}\times 2Y)\mid \mathrm{supp}\,u\subseteq {\mathbb{R}}\times Y
\right \}  ,
\end{split}
\end{align}
where one uses the product structure on ${\mathbb{R}}\times 2Y$. The inner
product of $L^{2}({\mathbb{R}}\times Y)$ yields an identification of the
dual space of $H^{s}_{p}({\mathbb{R}}\times Y)$ with
$\Hcirc ^{-s}_{p^{\prime }}({\mathbb{R}}\times Y)$, where
$1/p+1/p^{\prime }=1$.

We let ${\mathscr{C}}^{\infty }(Y)$ be the space of smooth (up to and including
the boundary) functions on $Y$, and
\begin{equation*}
{\mathscr{C}}^{\infty }_{c}({\mathbb{R}}\times Y)={\mathscr{C}}^{\infty }_{c}({
\mathbb{R}},{\mathscr{C}}^{\infty }(Y)),\qquad {\mathscr{S}}({\mathbb{R}}
\times Y)={\mathscr{S}}({\mathbb{R}},{\mathscr{C}}^{\infty }(Y)).
\end{equation*}
We extend the map
\begin{equation*}
S_{\gamma }:{\mathscr{C}}^{\infty }_{c}({\mathbb{R}}_{+}\times Y)
\longrightarrow {\mathscr{C}}^{\infty }_{c}({\mathbb{R}}\times Y),
\qquad (S_{\gamma }u)(r,y)=e^{(\gamma -1/2)r}u(e^{-r},y)
\end{equation*}
to the dual $($distribution$)$ spaces. Then we define
\begin{equation*}
{\mathcal{H}}^{s,\gamma }_{p}({\mathbb{R}}_{+}\times Y):= S^{-1}_{
\gamma -\frac{\mathrm{dim}\,Y}{2}}(H^{s}_{p}({\mathbb{R}}\times Y)),
\qquad \gamma \in {\mathbb{R}},
\end{equation*}
with the canonically induced norm; analogously we define
$\cHcirc ^{s,\gamma }_{p}({\mathbb{R}}_{+}\times Y)$. Moreover, we let
${\mathcal{T}}^{\gamma }({\mathbb{R}}_{+}\times Y)$ be the pre-image of
${\mathscr{S}}({\mathbb{R}}\times Y)$ under
$S_{\gamma -\frac{\mathrm{dim}\,Y}{2}}$.

Note that ${\mathcal{H}}^{s,\gamma }_{p}({\mathbb{R}}_{+}\times Y)$ for
$s\in {\mathbb{N}}_{0}$ and $1\le p<+\infty $ consists of all functions
$u(x,y)$ such that
\begin{equation}
\label{eq:conesob}
x^{\frac{\mathrm{dim}\,Y+1}{2}-\gamma } (x\partial _{x})^{j}\partial ^{\alpha }_{y} u(x,y) \in L^{p}\Big ({\mathbb{R}}_{+}\times Y,\frac{dxdy}{x}
\Big ), \qquad j+|\alpha |\le s.
\end{equation}

\subsubsection{Function spaces on ${\mathbb{D}}$ and ${\mathbb{B}}$}
\label{sec8.1.1}

\begin{definition}
${\mathcal{H}}^{s,\gamma }_{p}({\mathbb{D}})$ denotes the space of all
$u\in H^{s}_{p,\mathrm{loc}}({\mathbb{D}}_{\mathrm{reg}})$ such that
$\omega u\in {\mathcal{H}}^{s,\gamma }_{p}({\mathbb{R}}_{+}\times Y)$, equipped
with the norm
\begin{equation*}
\|u\|^{2}_{{\mathcal{H}}^{s,\gamma }_{p}({\mathbb{D}})}=\|\omega u\|^{2}_{{
\mathcal{H}}^{s,\gamma }_{p}({\mathbb{R}}_{+}\times Y)} +\|(1-\omega )u
\|^{2}_{H^{s}_{p}(2{\mathbb{D}})},
\end{equation*}
where $2{\mathbb{D}}$ is the double of ${\mathbb{D}}$ obtained by gluing
two copies of ${\mathbb{D}}$ along $\{0\}\times Y$. Analogously one defines
the space $\cHcirc ^{s,\gamma }_{p}({\mathbb{D}})$.
\end{definition}

The inner product of ${\mathcal{H}}^{0,0}_{2}({\mathbb{D}})$ allows for the
identification
\begin{equation}
\label{eq:dual}
{\mathcal{H}}^{s,\gamma }_{p}({\mathbb{D}})^{\prime }=\cHcirc ^{-s,-\gamma }_{p^{\prime }}({\mathbb{D}}).
\end{equation}

\begin{definition}
Let ${\mathscr{C}}^{\infty ,\gamma }({\mathbb{D}})$ denote the space of all
$u\in {\mathscr{C}}^{\infty }({\mathbb{D}}_{\mathrm{reg}})$ such that
$\omega u\in {\mathcal{T}}^{\gamma }({\mathbb{R}}_{+}\times Y)$. Moreover,
${\mathscr{C}}^{\infty ,\infty }(\mathbb{D}) :=\mathop{
\text{\large $\cap $}}_{\gamma \in \mathbb{R}}{\mathscr{C}}^{\infty ,
\gamma }(\mathbb{D})$.
\end{definition}

Replacing $Y$ by $\partial Y$ and ${\mathbb{D}}$ by ${\mathbb{B}}$ one obtains
the spaces ${\mathcal{H}}^{s,\gamma }_{p}({\mathbb{B}})$,
${\mathscr{C}}^{\infty ,\gamma }({\mathbb{B}})$, and
${\mathscr{C}}^{\infty ,\infty }({\mathbb{B}})$. Interpolation then furnishes
the Besov spaces ${\mathcal{H}}^{s,\gamma }_{p,q}({\mathbb{B}})$ on
${\mathbb{B}}$.

\subsubsection{Function spaces on model cones}
\label{sec8.1.2}

Recall that we write $Y^{\wedge }={\mathbb{R}}_{+}\times Y$.

\begin{definition}
${\mathscr{S}}^{\gamma }(Y^{\wedge })$ is the space of all
$u\in {\mathscr{C}}^{\infty }(Y^{\wedge })$ such that
$\omega u \in {\mathcal{T}}^{\gamma }(Y^{\wedge })$ and
$(1-\omega )u\in {\mathscr{S}}({\mathbb{R}}\times Y)$. Moreover,
${\mathscr{S}}^{\infty }(Y^{\wedge }) :=\mathop{\text{\large $\cap $}}_{
\gamma \in \mathbb{R}}{\mathscr{S}}^{\gamma }(Y^{\wedge })$.
\end{definition}

Let $(U_{1},\kappa _{1}),\ldots ,(U_{N},\kappa _{N})$ be an atlas of the
manifold $2Y$. Using the charts
$({\mathbb{R}}\times U_{j},\widehat{\kappa }_{j})$ with
$\widehat{\kappa }_{j}(r,y)=(r,\langle r\rangle \kappa _{j}(y))$, allows
to define the Sobolev spaces
$H^{s}_{p,\asymp }({\mathbb{R}}\times 2Y)$ as the pullback of the standard
Sobolev spaces, see \cite[Section 4.2]{SS_2} for details. As in \reftext{\eqref{eq:SobZyl}} one obtains the spaces
$H^{s}_{p,\asymp }({\mathbb{R}}\times Y)$ and
$\Hcirc ^{s}_{p,\asymp }({\mathbb{R}}\times Y)$.

\begin{definition}
${\mathcal{K}}^{s,\gamma }_{p}(Y^{\wedge })$ denotes the space of all
$u\in H^{s}_{p,loc}(Y^{\wedge })$ such that
$\omega u \in {\mathcal{H}}^{s,\gamma }_{p}({\mathbb{D}})$ and
$(1-\omega )u \in H^{s}_{p,\asymp }({\mathbb{R}}\times Y)$; an analogous
construction yields $\cKcirc ^{s,\gamma }_{p}(Y^{\wedge })$.
\end{definition}

Similarly one defines the spaces
${\mathscr{S}}^{\gamma }(\partial Y^{\wedge })$,
${\mathscr{S}}^{\infty }(\partial Y^{\wedge })$, and
${\mathcal{K}}^{s,\gamma }_{p}(\partial Y^{\wedge })$ with
$\partial Y^{\wedge }={\mathbb{R}}_{+}\times \partial Y$. Interpolation yields
the Besov spaces ${\mathcal{K}}^{s,\gamma }_{p,q}(\partial Y^{\wedge })$ on
$\partial Y^{\wedge }$.

\subsection{Elements of a parameter-dependent edge type
calculus on manifolds with boundary and conical singularities}
\label{app:cone}

We recall a few basic facts concerning a parameter-dependent calculus on
conic manifolds with boundary. This is a version of the boundary edge calculus
as developed in \cite[Section 7.2]{HS} and \cite[Section 4]{KS} for the
case where the edge is a single point.

By ${\mathcal{B}}^{\mu ,d}(Y;\Sigma )$ we denote the parameter-dependent
elements of order $\mu $ and type $d$ in Boutet de Monvel's calculus with
parameter space $\Sigma $. See Section 9 in \cite{CSS1} for a concise presentation.
We write
$\Gamma _{\beta }= \{z\in {\mathbb{C}}\mid \revtex z =\beta \}$.

For hermitian vector bundles $E_{0},E_{1}$ over $Y$ and
$F_{0},F_{1}$ over $\partial Y$,
$\gamma _{0} , \gamma _{1}\in {\mathbb{R}}$ and
$\ell _{0},\ell _{1}\in {\mathbb{N}}_{0}$ we write $ E_{j}^{\wedge }$ and
$F_{j}^{\wedge }$ for the pullback of $(E_{j})_{|x=0}$ and
$(F_{j})_{|x=0}$ to $Y^{\wedge }$ and $\partial Y^{\wedge }$ and define the
spaces
\begin{align*}
\pmb{{\mathcal{K}}}_2^{s,\gamma _{j}}(Y^{\wedge }; E_{j}^{\wedge },F_{j}^{\wedge }, {\mathbb{C}}^{\ell _{j}}) &= {\mathcal{K}}_2^{s,\gamma _{j}}(Y^{\wedge }; E_{j}^{\wedge }) \oplus {\mathcal{K}}_2^{s,\gamma _{j}-1/2}((
\partial Y)^{\wedge }; F_{j}^{\wedge })\oplus {\mathbb{C}}^{\ell _{j}},
\\
\pmb{{\mathscr{S}}}^{\gamma _{j}}(Y^{\wedge }; E_{j}^{\wedge },F_{j}^{\wedge },{\mathbb{C}}^{\ell _{j}}) &= {\mathscr{S}}^{\gamma _{j}}(Y^{\wedge }; E_{j}^{\wedge }) \oplus {\mathscr{S}}^{\gamma _{j}-1/2}((
\partial Y)^{\wedge }; F_{j}^{\wedge })\oplus {\mathbb{C}}^{\ell _{j}}.
\end{align*}

\begin{definition}%
For $\nu \in {\mathbb{Z}}$, $d\in {\mathbb{N}}_{0}$,
$\gamma _{0},\gamma _{1}\in {\mathbb{R}}$ we denote by
$R^{\mu ,0}_{G}({\mathbb{D}},\Sigma ; \gamma _{0},\gamma _{1})$ the space
of all operator families $g(\sigma )$, $\sigma \in \Sigma $ such that,
for some $\varepsilon >0$,
\begin{eqnarray*}
g(\sigma ) &\in & S^{\nu }_{cl} (\Sigma ; \pmb{{\mathcal{K}}}_2^{0,\gamma _{0}}(Y^{\wedge }; E_{0}^{\wedge },F_{0}^{\wedge },{\mathbb{C}}^{\ell _{0}}),
\pmb{{\mathscr{S}}}^{\gamma _{1}+\varepsilon }(Y^{\wedge }; E_{1}^{\wedge },F_{1}^{\wedge },{\mathbb{C}}^{\ell _{1}})),
\\
g^{*}(\sigma ) &\in & S^{\nu }_{cl} (\Sigma ; \pmb{{\mathcal{K}}}_2^{0,-
\gamma _{1}}(Y^{\wedge }; E_{1}^{\wedge },F_{1}^{\wedge },{\mathbb{C}}^{
\ell _{1}}), \pmb{{\mathscr{S}}}^{-\gamma _{0}+\varepsilon }(Y^{\wedge };
E_{0}^{\wedge },F_{0}^{\wedge }, {\mathbb{C}}^{\ell _{0}}))
\end{eqnarray*}
where the asterisk denotes the pointwise adjoint with respect to the inner
product in
$\pmb{{\mathcal{K}}}_2^{0,0}(Y^{\wedge };\break  E_{0}^{\wedge },F_{0}^{\wedge },{
\mathbb{C}}^{\ell _{0}})$ and
$\pmb{{\mathcal{K}}}_2^{0,0}(Y^{\wedge }; E_{1}^{\wedge },F_{1}^{\wedge },{
\mathbb{C}}^{\ell _{1}})$, respectively.

For $d>0$, the space
$R^{\mu ,d}_{G}({\mathbb{D}},\Sigma ; \gamma _{0},\gamma _{1})$ consists
of all operator families of the form
\begin{equation*}
\sum _{j=0}^{d} g_{j}(\sigma )\partial _{\nu }^{j}: \pmb{{\mathcal{K}}}_2^{s,
\gamma _{0}}(Y^{\wedge }; E_{0}^{\wedge },F_{0}^{\wedge },{\mathbb{C}}^{
\ell _{0}}) \to \pmb{{\mathscr{S}}}^{\gamma _{1}+\varepsilon }(Y^{\wedge }; E_{1}^{\wedge },F_{1}^{\wedge },{\mathbb{C}}^{\ell _{1}})),
\quad s>d-1/2,
\end{equation*}
with
$g_{j}\in R_{G}^{\nu ,0} ({\mathbb{D}}, \Sigma; \gamma _{0}, \gamma _{1})$ and the
normal derivative $\partial _{\nu }$.
\end{definition}

In the previous definition, the data $E_{j}$, $F_{j}$ and
${\mathbb{C}}^{\ell _{j}}$ should be part of the notation; we have omitted
them here for better legibility.

Similarly to the above notation we let
\begin{align*}
\pmb{{\mathcal{H}}}_2^{s,\gamma _{j}}({\mathbb{D}}; E_{j},F_{j},{
\mathbb{C}}^{\ell _{j}}) &= {\mathcal{H}}_2^{s,\gamma _{j}}({\mathbb{D}}; E_{j})
\oplus {\mathcal{H}}_2^{s,\gamma _{j}-1/2}({\mathbb{B}}; F_{j})\oplus {
\mathbb{C}}^{\ell _{j}}
\\
\pmb{{\mathscr{C}}}^{\infty ,\gamma _{j}}({\mathbb{D}},E_{j},F_{j}, {
\mathbb{C}}^{\ell _{j}}) &={\mathscr{C}}^{\infty ,\gamma _{j}}({
\mathbb{D}},E_{j}) \oplus {\mathscr{C}}^{\infty ,\gamma _{j}-1/2}({
\mathbb{B}},F_{j}) \oplus {\mathbb{C}}^{\ell _{j}}.
\end{align*}

\begin{definition}%
\label{CGD2}
The space
$C^{\nu ,0}_{G}({\mathbb{D}},\Sigma ;\gamma _{0},\gamma _{1})$ consists
of all operator-families $g(\sigma )$ of the form
\begin{equation*}
g(\sigma )=\omega _{1}\,a(\sigma )\,\omega _{0}+r(\sigma ),
\end{equation*}
where $\omega _{0},\omega _{1}$ are cut-off functions,
$a\in R^{\nu ,0}_{G}({\mathbb{D}},\Sigma ;\gamma _{0},\gamma _{1})$, and
$r$ has an integral kernel in
\begin{equation*}
{{\mathscr{S}}}(\Sigma ,\pmb{{\mathscr{C}}}^{\infty ,\gamma _{1}+
\varepsilon }({\mathbb{D}},E_{1},F_{1},{\mathbb{C}}^{\ell _{1}}) \,
\widehat{\otimes }_{\pi }\, \pmb{{\mathscr{C}}}^{\infty ,-\gamma _{0}+
\varepsilon }({\mathbb{D}},E_{0},F_{0}. {\mathbb{C}}^{\ell _{0}}).
\end{equation*}
for some $\varepsilon =\varepsilon (g)>0$. For $d\ge 1$,
$C^{\nu ,d}_{G}({\mathbb{D}},\Sigma ;\gamma _{0},\gamma _{1})$ is the space
of all operator families
$\sum _{j=0}^{d} g_{j}(\sigma )\partial _{\nu }^{j}$ with
$g_{j}\in C^{\nu ,0}_{G}({\mathbb{D}},\Sigma ;\gamma _{0},\gamma _{1})$.

We define the principal operator-valued symbol
$\sigma ^{\nu }_{\wedge }(g)$ of $g$ to be the principal operator-valued symbol
of $a$.
\end{definition}

\begin{definition}%
\label{real38}
A \emph{holomorphic Mellin symbol} of order $\mu $ and type $d$, depending
on the parameter $\sigma \in \Sigma $, is a holomorphic function
$h:{\mathbb{C}}\to {\mathcal{B}}^{\mu ,d}(Y;\Sigma )$ such that
\begin{equation*}
\delta \mapsto h(\delta +i\tau ,\sigma ):{\mathbb{R}}\longrightarrow {
\mathcal{B}}^{\mu ,d}(Y;{\mathbb{R}}_{\tau }\times \Sigma )
\end{equation*}
is continuous. We denote the space of all such symbols by
$M^{\mu ,d}_{\mathcal{O}}(Y;\Sigma )$ and set
\begin{equation*}
M^{\mu ,d}_{\mathcal{O}}(\overline{{\mathbb{R}}}_{+}\times Y;\Sigma ) ={
\mathscr{C}}^{\infty }(\overline{{\mathbb{R}}}_{+})\widehat{\otimes }_{\pi }M^{
\mu ,d}_{\mathcal{O}}(Y;\Sigma ).
\end{equation*}
The space $M^{-\infty ,d}_{\gamma }(Y)$ of \emph{smoothing Mellin symbols}
of type $d$ consists of all maps $h_{0}$ which are holomorphic in an
$\varepsilon $-strip around the line $\Gamma _{\frac{n+1}{2}-\gamma }$,
$\varepsilon $ arbitrarily small, taking values in
${\mathcal{B}}^{-\infty ,d}(Y)$.
\end{definition}

For every $\sigma \in \Sigma $, a holomorphic Mellin symbol $h$ and a smoothing
Mellin symbol $h_{0}$ define an operator
\begin{equation*}
\opvtex _{M}^{\gamma -n/2}(h+h_{0})(\sigma ):
\begin{matrix}
{\mathscr{C}}^{\infty }_{c}({\mathbb{R}}_{+}\times Y,E_{0}^{\wedge })
\\
\oplus
\\
{\mathscr{C}}^{\infty }_{c}({\mathbb{R}}_{+}\times \partial Y, F_{0}^{\wedge })%
\end{matrix}
\longrightarrow
\begin{matrix}
{\mathscr{C}}^{\infty }({\mathbb{R}}_{+}\times Y, E_{1}^{\wedge })
\\
\oplus
\\
{\mathscr{C}}^{\infty }({\mathbb{R}}_{+}\times \partial Y,F_{1}^{\wedge })%
\end{matrix}
\end{equation*}
by
\begin{equation}
\label{real3E}
[\opvtex _{M}^{\gamma -n/2}(h+h_{0})(\sigma )u](x)=\int _{\Gamma _{
\frac{1}{2}-\gamma }} x^{-z}(h(x,z,x\sigma )+h_{0}(z))(\widehat{u})(z)
\,\dbar z,
\end{equation}
where $\widehat{u}$ denotes the Mellin transform.

In the sequel, we will consider
$\opvtex _{M}^{\gamma -n/2}(h+h_{0})$ also as an operator
\begin{align}
\label{eq:ident}
\opvtex _{M}^{\gamma -n/2}(h+h_{0}): \pmb{{\mathcal{K}}}_2^{s,\gamma }(Y^{\wedge }; E_{0}^{\wedge },F_{0}^{\wedge },{\mathbb{C}}^{\ell _{0}})
\longrightarrow \pmb{{\mathcal{K}}}_2^{s-\nu ,\gamma -\nu }(Y^{\wedge }; E_{1}^{\wedge },F_{1}^{\wedge },{\mathbb{C}}^{\ell _{1}})),
\end{align}
by identifying $\opvtex _{M}^{\gamma -n/2}(h+h_{0})$ with the operator-matrix
$
\begin{pmatrix}
\opvtex _{M}^{\gamma -n/2}(h+h_{0})&0
\\
0&0
\end{pmatrix}
$.

\begin{definition}
$C^{\nu ,d}({\mathbb{D}},\Sigma ;\gamma , \gamma -\nu )$ denotes the space
of all operator-valued symbols
\begin{align}
\label{edgesymbol}
\begin{split}
b(\sigma ) = & \ x^{-\nu }\omega \opvtex _{M}^{\gamma -n/2}(h)(\sigma )
\omega _{1} + (1-\omega )b_{\mathrm{int}}(\sigma ) (1-\omega _{2})
\\
&+ x^{-\nu }\tilde{\omega }(x[\sigma ])\opvtex _{M}^{\gamma -n/2} (h_{0})
\tilde{\omega }(x[\sigma ]) + g(\sigma ) ,
\end{split}
\end{align}
where $h$ and $h_{0}$ are as above,
$b_{\mathrm{int}} \in {\mathcal{B}}^{\nu ,d}(2{\mathbb{D}};\Sigma )$,
$\omega , \omega _{1}, \omega _{2}$ and $\tilde{\omega }$ are cut-of functions
near $x=0$, $\omega \omega _{1} = \omega $,
$\omega \omega _{2} = \omega _{2}$, and
$g\in C_{G}^{\nu ,d}({\mathbb{D}},\Sigma ; \gamma , \gamma -\nu )$.

The principal edge symbol associated with the operator-valued symbol
$b$ in \reftext{\eqref{edgesymbol}} is
\begin{align}
\label{princedgesymbol}
\begin{split}
\sigma ^{\nu }_{\wedge }(b)(\sigma ) = &\ x^{-\nu } \opvtex _{M}^{\gamma -n/2}(h_{|x=0})(
\sigma )
\\
&+ x^{-\nu }\tilde{\omega }(x|\sigma |)\opvtex _{M}^{\gamma -n/2} (h_{0})
\tilde{\omega }(x|\sigma |)+ \sigma ^{\nu }_{\wedge }(g)(\sigma ),
\end{split}
\end{align}
where $\sigma ^{\nu }_{\wedge }(g)$ is the principal operator-valued symbol
of $g$; it is a map
\begin{equation*}
\pmb{{\mathcal{K}}}_2^{s,\gamma }(Y^{\wedge }; E_{0},F_{0},{\mathbb{C}}^{
\ell _{0}})\longrightarrow \pmb{{\mathcal{K}}}_2^{s-\nu ,\gamma -\nu }(Y^{\wedge }; E_{1},F_{1},{\mathbb{C}}^{\ell _{1}})).
\end{equation*}
\end{definition}

This definition follows the approach in \cite[Section 4.6]{KS} and uses
the alternative representation of the symbols in
\cite[Theorem 4.6.29]{KS}, going back to \cite{GSS}.

Localized to any open $U\subseteq {\mathbb{D}}_{\mathrm{reg}}$, the operator
$b(\sigma )$ is given by a parameter-dependent operator
$B_{U}\in {\mathcal{B}}^{\mu ,d}(2{\mathbb{D}};\Sigma )$ (here we make the
same identification as in \reftext{\eqref{eq:ident}}). The principal symbols of these
patch to a smooth homogenous interior principal symbol and a smooth homogeneous
boundary symbol on $T^{*}{\mathbb{D}}_{\mathrm{reg}}\setminus 0 $ and
$T^{*}{\mathbb{B}}_{\mathrm{reg}}\setminus 0$, respectively. Similarly as for
the symbols introduced after \reftext{\eqref{eq:fullbvp}}, these degenerate as
$x\to 0$ and can be rescaled as explained in \reftext{\eqref{eq:rescaledprinc}} and \reftext{\eqref{rescaledbdry}}.

We call $b$ elliptic, if the interior principal symbol, the rescaled interior
principal symbol, the boundary symbol, the rescaled boundary symbol and
the operator-valued symbol $\sigma _{\wedge }^{\nu }(\sigma )$ are all invertible.
Following \cite[Section 4.5]{KS} we obtain:

\begin{theorem}
Let $b\in C^{\nu ,d}({\mathbb{D}},\Sigma ;\gamma , \gamma -\nu )$,
$d\le \max \{\nu ,0\}$, be elliptic. Then there exists a parametrix
$c\in C^{-\nu ,d'}({\mathbb{D}},\Sigma ;\gamma -\nu , \gamma )$,
$d'\le \max \{-\nu ,0\}$ which inverts $b$ modulo smoothing Green operators,
i.e.
\begin{equation*}
bc-1\in C_{G}^{-\infty ,d'}({\mathbb{D}};\Sigma ,\gamma -\nu ,\gamma -
\nu )\text{ and } cb-1\in C_{G}^{-\infty ,d}({\mathbb{D}};\Sigma ,
\gamma ,\gamma ) .
\end{equation*}
\end{theorem}


\end{document}